\documentclass[10pt]{amsart}
\usepackage{graphicx}
\usepackage{amssymb}
\usepackage{enumerate}
\usepackage{mathrsfs}

\numberwithin{equation}{section}
\numberwithin{figure}{section}
\theoremstyle{plain}
\newtheorem{thm}{Theorem}[section]
\theoremstyle{definition}

\newtheorem*{claim}{Claim}
\theoremstyle{plain}
\newtheorem{prop}[thm]{Proposition}
\theoremstyle{remark}
\newtheorem{rem}[thm]{Remark}
\theoremstyle{plain}
\newtheorem{cor}[thm]{Corollary}
\theoremstyle{plain}
\newtheorem{lem}[thm]{Lemma}
\theoremstyle{remark}
\newtheorem*{rem*}{Remark}

\newcommand{\norm}[1]{\left\Vert#1\right\Vert}
\newcommand{\abs}[1]{\left\vert#1\right\vert}
\newcommand{\set}[1]{\left\{#1\right\}}
\newcommand{\pset}[1]{\left(#1\right)}
\newcommand{\Real}{\mathbb{R}}
\newcommand{\Zeit}{\mathbb Z}
\newcommand{\Cay}{\mathbb O}
\newcommand{\Cpx}{\mathbb{C}}

\newcommand{\eps}{\varepsilon}

\newcommand{\sph}{\mathbb{S}}
\newcommand{\To}{\rightarrow}

\newcommand{\bk}[1]{\langle #1 \rangle}
\newcommand{\RP}{\mathbb{R} \mathrm{P}}
\newcommand{\calB}{\mathcal B}

\newcommand{\Bm}{B_{-}}
\newcommand{\Bp}{B_{+}}

\newcommand{\Tr}{\mathrm{tr}}

\newcommand{\Span}{\mathrm{span}}
\newcommand{\diag}{\mathrm{diag}}
\newcommand{\Ad}{\mathrm{Ad}}
\newcommand{\Id}{\mathrm{Id}}

\newcommand{\sfGt}{\mathsf G_2}
\newcommand{\SO}{\mathsf{SO}}
\newcommand{\sfO}{\mathsf{O}}
\newcommand{\SU}{\mathsf{SU}}
\newcommand{\U}{\mathsf{U}}
\newcommand{\Km}{\mathsf{K}^{-}}
\newcommand{\Kp}{\mathsf{K}^{+}}
\newcommand{\sfH}{\mathsf{H}}
\newcommand{\sfG}{\mathsf{G}}
\newcommand{\barKm}{\bar{\mathsf{K}}^{-}}
\newcommand{\barKp}{\bar{\mathsf{K}}^{+}}

\newcommand{\frakg}{\mathfrak{g}}
\newcommand{\frakh}{\mathfrak{h}}
\newcommand{\frakp}{\mathfrak{p}}
\newcommand{\frakk}{\mathfrak{k}}
\newcommand{\frakm}{\mathfrak{m}}
\newcommand{\so}{\mathfrak{so}}
\newcommand{\su}{\mathfrak{su}}

\title[Fake $\Real\mathrm{P}^{13}$ with cohomogeneity one actions]{Fake $13$-projective spaces \\ with cohomogeneity one actions}

\author{Chenxu He}
\address{Department of Mathematics, University of California, Riverside, CA 92521}
\email{chenxuhe@math.ucr.edu}
\urladdr{https://sites.google.com/site/hechenxu/home}
\author{Priyanka Rajan}
\address{Department of Mathematics, University of California, Riverside, CA 92521}
\email{rajan@math.ucr.edu}
\urladdr{https://sites.google.com/site/priyankarrajangeometry/home}


\begin{document}
\maketitle

\begin{abstract}
We show that some embedded standard $13$-spheres in Shimada's exotic $15$-spheres have $\Zeit_2$ quotient spaces, $P^{13}$s, that are fake real $13$-dimensional projective spaces, i.e., they are homotopy equivalent, but not diffeomorphic to the standard $\RP^{13}$. As observed by F. Wilhelm and the second named author in \cite{RajanWilhelm}, the Davis $\SO(2)\times \sfGt$ actions on Shimada's exotic $15$-spheres descend to the cohomogeneity one actions on the $P^{13}$s. We prove that the $P^{13}$s are diffeomorphic to well-known $\Zeit_2$ quotients of certain Brieskorn varieties, and that the Davis $\SO(2)\times \sfGt$ actions on the $P^{13}$s are equivariantly diffeomorphic to well-known actions on these Brieskorn quotients. The $P^{13}$s are octonionic analogues of the Hirsch-Milnor fake $5$-dimensional projective spaces, $P^{5}$s. K. Grove and W. Ziller showed that the $P^{5}$s admit metrics of non-negative curvature that are invariant with respect to the Davis $\SO(2)\times \SO(3)$-cohomogeneity one actions. In contrast, we show that the $P^{13}$s do not support $\SO(2)\times \sfGt$-invariant metrics with non-negative sectional curvature.
\end{abstract}

\tableofcontents

\newpage

\section{Introduction}
A \emph{fake} real projective space is a manifold homotopy equivalent, but not diffeomorphic, to the standard real projective space. Equivalently, it is the orbit space of a free \emph{exotic} involution on a sphere. A free involution is called exotic, if it is not conjugate by a diffeomorphism to the standard antipodal map on the sphere. The first examples of such exotic involutions were constructed by Hirsch and Milnor on $\sph^5$ and $\sph^6$, see \cite{HirschMilnor}. They are restrictions of certain free involutions on the images of embedded standard $5$- and $6$-spheres in Milnor's exotic spheres  \cite{Milnor7sphere}. Thus the quotient spaces of such embedded $\sph^5$ and $\sph^6$ are homotopy equivalent, but not diffeomorphic, to the standard real projective spaces. 

The analogous exotic $15$-spheres $\Sigma^{15}$s were constructed by N. Shimada in \cite{Shimada} as certain $7$-sphere bundles over the $8$-sphere. The antipodal map on the $7$-sphere fiber defines a natural involution $T$ on the $\Sigma^{15}$s. In \cite{RajanWilhelm}, F. Wilhelm and the second named author observed that the images of certain embedded standard $13$- and $14$-spheres in $\Sigma^{15}$s  are invariant under the involution, and thus the quotient spaces are homotopy equivalent to the standard $13$- and $14$-real projective spaces. Our first main result is the diffeomorphism classification of the quotients. In particular we show the following
\begin{thm}\label{thm:introP13kfake}
The quotient spaces of the embedded $13$-spheres in certain Shimada's spheres $\Sigma^{15}$s are fake real projective spaces, i.e., they are homotopy equivalent, but not diffeomorphic to the standard $13$-projective space. 
\end{thm}

\begin{rem}
(a) In \cite{RajanWilhelm}, they showed that the quotients of the embedded $14$-spheres in some $\Sigma^{15}$s are not diffeomorphic to the standard $\RP^{14}$ following the Hirsch-Milnor argument.

(b) They also observed that the Hirsch-Milnor's argument breaks down in the case of the embedded $13$-spheres as there is an exotic $14$-sphere in contrast to the $6$-sphere.
\end{rem}

Our proof of diffeomorphism classification is through the study of the so called Davis action of $\sfG = \SO(2) \times \sfGt$ on Shimada's exotic $15$-spheres, where $\sfGt$ is the simple exceptional Lie group as the automorphism group of the octonions $\Cay$. For each odd integer $k$, denote $\Sigma_k^{15}$ the total space of the $7$-sphere bundle over the $8$-sphere, with the Euler class $[\sph^8]$ and the second Pontrjagin class $6 k [\sph^8]$ where $[\sph^8]$ is the standard generator of the cohomology group $H^8(\sph^8)$. Shimada showed that each $\Sigma^{15}_k$ is homeomorphic to the standard $15$-sphere, but not diffeomorphic if $k^2 \not\equiv 1 \mod 127$, see \cite{Shimada}. In \cite{Davis}(or see Section 2.1), using the octonion algebra, M. Davis introduced the actions of $\sfG $ on $\Sigma_k^{15}$s such that $\sfGt$ acts diagonally on the $7$-sphere fiber and the $8$-sphere base, whereas $\SO(2)$ acts via M\"obius transformation. It is observed in \cite{RajanWilhelm}, that the Davis action on $\Sigma^{15}_k$ leaves the image $\sph^{13}_k$ of the embedded $13$-sphere invariant and commutes with the involution $T$. Thus the restricted action on $\sph^{13}_k$ descends to the quotient space $P^{13}_k = \sph^{13}_k/T$. They also observed that the $\sfG$-actions on $\sph^{13}_k$ and $P^{13}_k$ are cohomogeneity one, i.e., the orbit spaces are one dimensional. On the other hand, for the cohomogeneity one actions on the homotopy spheres, aside from linear actions on the standard spheres, there are families of non-linear actions \cite{Straume}. They are examples given by the $2n-1$ dimensional Brieskorn varieties $M_d^{2n-1}$, which are defined by the equations 
\[
z_0^d +z_1^2 + \ldots + z_n^2 = 0 \quad \text{and} \quad \abs{z_0}^2 + \abs{z_1}^2 + \ldots + \abs{z_n}^2 = 1.
\]
The Brieskorn varieties carry cohomogeneity one actions  by $\SO(2) \times\SO(n)$ via 
\[
(e^{i\theta}, A)\left(z_0, z_1, \ldots, z_n\right) = \left(e^{2i \theta} z_0, e^{-i d\theta }A (z_1, \ldots, z_n)^t\right)
\]
with $A\in \SO(n)$. A natural involution, denoted by $I$, is defined by $I(z_0, z_1, \ldots, z_n) = (z_0, -z_1, \ldots, -z_n)$. It is clear that the involution has no fixed point and commutes with the $\SO(2) \times \SO(n)$-action; and thus the quotient space $N_d^{2n-1} = M_d^{2n-1}/I$ admits a cohomogeneity one action by $\SO(2) \times \SO(n)$. Note that when $n= 7$, the actions on $M_d^{13}$ and $N_d^{13}$ restricted to the group $\sfG = \SO(2) \times \sfGt$ are also cohomogeneity one. We have the following
\begin{thm}\label{thm:introBrieskornvarieties}
For each odd integer $k$, the $\sfG$-manifolds: the $13$-sphere $\sph_k^{13}$ and the Brieskorn variety $M_k^{13}$, with $\sfG = \SO(2) \times \sfGt$ are equivariantly diffeomorphic, and so are the quotient spaces $P_k^{13} = \sph_k^{13}/T$ and $N_k^{13}= M_k^{13}/ I$.
\end{thm}

\begin{rem}
Theorem \ref{thm:introP13kfake} follows from Theorem \ref{thm:introBrieskornvarieties} above and the diffeomorphism classification of $N_d^{2n-1}$ in \cite{AtiyahBott} and \cite{Giffen} (or see Section 2.2).
\end{rem}
\begin{rem}
The space $P^{13}_1$, i.e., $k =1$, is diffeomorphic to the standard $\RP^{13}$ from the construction in \cite{Shimada} and \cite{RajanWilhelm}. From Theorem \ref{thm:introBrieskornvarieties} above, the known diffeomorphism classification of $N^{13}_k$ implies that there are $64$ different oriented diffeomorphism types of $P_k^{13}$s. 
\end{rem}

\begin{rem}
(a) The Davis actions of $\SO(2) \times \sfGt$ on Shimada's exotic spheres $\Sigma^{15}_k$s can be viewed as the octonionic analogs of the $\SO(2) \times \SO(3)$ actions on Milnor's exotic spheres $\Sigma^7$s found in the same paper \cite{Davis}. Note that $\SO(3)$ is the automorphism group of the quaternions, and a special case of the $\SO(2)\times \SO(3)$ actions on a certain $\Sigma^7$ was found in \cite{GromollMeyer}. 

(b) The Davis actions of $\SO(2) \times \SO(3)$ on Milnor's exotic spheres also leave the images of the embedded $5$-sphere invariant, and hence induce cohomogeneity one actions on the Hirsch-Milnor's fake $5$-projective spaces as observed in \cite{RajanWilhelm}. These actions are equivariantly diffeomorphic to those on the Brieskorn varieties $N^5_d$'s, which was first discovered by E. Calabi(unpublished, cf. \cite[p. 368]{HsiangHsiang})
\end{rem}

\begin{rem}
In \cite{ADPR}, U. Abresch, C. Dur\'an, T. P\"uttmann and A. Rigas gave a geometric construction of free exotic involutions on the Euclidean sphere $\sph^{13}$ using the wiedersehen metric on the Euclidean sphere $\sph^{14}$. Thus the quotient spaces are fake $13$-projective spaces. Moreover, in \cite{DuranPuettmann}, Dur\'an and P\"uttmann provided an explicit nonlinear action of $\sfO(2) \times \sfGt$ on the Euclidean sphere $\sph^{13}$, and showed that it is equivariantly diffeomorphic to the Brieskorn variety $M^{13}_{3}$.
\end{rem}

\smallskip

The second part of this paper is the study of the curvature properties of the invariant metrics on $\sph^{13}_k$ and $P^{13}_k$ with $\sfG = \SO(2) \times\sfGt$. Since any invariant metric on the quotient space $P_k^{13}$ can be lifted to an invariant metric on $\sph^{13}_k$, we restrict ourselves to the spheres $\sph^{13}_k$s, or equivalently $M_k^{13}$s. Note that $M^{13}_k$ and $M^{13}_{-k}$ are equivariantly diffeomorphic, and so we assume that $k \geq 1$. 

On a Riemannian manifold with cohomogeneity one action, the principal orbits are hypersurfaces, and there are precisely two non-principal orbits that have codimensions strictly bigger than one if the manifold is simply-connected. They are called singular orbits. In \cite{GroveZillerMilnorsphere}, K. Grove and W. Ziller constructed invariant metrics with non-negative sectional curvature on cohomogeneity one manifolds for which both singular orbits have codimension two. Particularly, their construction yields non-negatively curved metrics on $10$ of $14$ (unoriented) Milnor's spheres and all Hirsch-Milnor's fake $5$-projective spaces. However, not every cohomogeneity one manifold admits an invariant metric with non-negative curvature. The first examples were found by K. Grove, L. Verdiani, B. Wilking and W. Ziller in \cite{GVWZ}, and then generalized to a larger class in \cite{Heobstruction} by the first named author. The most interesting class in \cite{GVWZ} is the Brieskorn varieties $M_d^{2n-1}$. The Brieskorn variety $M_{d}^{2n-1}$ is homeomorphic to the sphere, if and only if, both $n$ and $d$ are odd. In \cite{GVWZ},  it is showed that for $n \geq 4$ and $d\geq 3$, $M_d^{2n-1}$ does not support an $\SO(2) \times \SO(n)$ invariant metric with non-negative curvature. In particular, there is no non-negatively curved $\SO(2) \times \SO(7)$ invariant metric on $M^{13}_d$, if $d\geq 3$. Since $\sfG$ is a proper subgroup in $\SO(2) \times \SO(7)$, there are more invariant metrics on $M^{13}_k$. One may suspect that there might be a chance to find an invariant metric with non-negative curvature. Nevertheless we show that the obstruction does appear even though the metric has a smaller symmetry group. 
\begin{thm}\label{thm:introobstruction}
For any odd integer $k \geq 3$, the Brieskorn variety $M^{13}_k$ does not support an $\SO(2) \times \sfGt$ invariant metric with non-negative curvature. 
\end{thm}
\begin{rem}
The techniques used to prove Theorem \ref{thm:introobstruction} are similar to those in \cite{GVWZ} and \cite{Heobstruction}. However the special feature of the Lie group $\sfGt$ and the strictly larger class of invariant metrics make the argument more involved.
\end{rem}

\begin{rem}
For the Brieskorn variety $M^{13}_d$ with $d\geq 4$ an even integer, the principal isotropy subgroup has a simpler form than the one in the odd case, see Remark \ref{rem:M13deven}. This leads to a much more complicated form of the invariant metrics in the even case, see Remark \ref{rem:M13deveninvariantmetrics}, which is not covered by our proof. So for an even integer $d\geq 4$, the question whether $M^{13}_d$ admits an $\SO(2) \times \sfGt$-invariant metric with non-negative curvature remains open. 
\end{rem}

From Theorems \ref{thm:introBrieskornvarieties} and \ref{thm:introobstruction}, we have the following
\begin{cor}
For any odd integer $k\geq 3$, the fake $13$-projective space $P^{13}_k$ does not support an $\SO(2) \times\sfGt$ invariant metric with non-negative curvature. 
\end{cor}
\begin{rem}
In contrast to the $P^{13}_k$s, it is observed by O. Dearricott that, following Grove-Ziller's construction, all fake Hirsch-Milnor's $5$-projective spaces admit $\SO(2) \times \SO(3)$ invariant metrics with non-negative curvature, see \cite[p. 334]{GroveZillerMilnorsphere}.
\end{rem}
\begin{rem}
As observed in \cite{SchwachTuschmann}, all $P^{13}_k$s and $\sph^{13}_k$s support even $\SO(2) \times \SO(7)$ invariant metrics that simultaneously have positive Ricci curvature and almost non-negative sectional curvature. For the invariant metrics with positive Ricci curvature alone, it also follows from the result in \cite{GroveZillerposRicci}.  A Riemannian manifold admits an almost non-negative sectional curvature if it collapses to a point with a uniform lower curvature bound. 
\end{rem}

From the classification of cohomogeneity one actions on homotopy spheres in \cite{Straume} by E. Straume, $M^{13}_k$s with $\sfG = \SO(2) \times \sfGt$ are the only nonlinear actions where the symmetry group does not have the form $\SO(2) \times \SO(n)$. Combining the classification in \cite{Straume}, the obstructions in \cite{GVWZ} and Theorem \ref{thm:introobstruction}, we have the following
\begin{cor}
For $n \geq 2$, let $\Sigma^n$ be a homotopy sphere. Suppose that $\Sigma^n$admits a non-negatively curved metric that is invariant under a cohomogeneity one action. Then either 
\begin{enumerate}
\item $\Sigma^n$ is equivariantly diffeomorphic to the standard sphere and the action is linear, or
\item $n= 5$, $\Sigma^5$ is the standard $5$-sphere and the non-linear action is given by $\SO(2) \times \SO(3)$ on the Brieskorn variety $M^5_k$, with $k\geq 3$ odd. 
\end{enumerate}
\end{cor}

\smallskip

We refer to the Table of Contents for the organization of the paper. Theorem \ref{thm:introBrieskornvarieties} is proved in Section 3, and Section 6 is the proof of Theorem \ref{thm:introobstruction}.

\medskip

\textbf{Acknowledgement.} It is a great pleasure to thank Frederick Wilhelm who has brought this problem to our attention, and  we had numerous discussions with him on this paper. We also thank Wolfgang Ziller for useful communications, and Karsten Grove for his interest.

\medskip
\section{Preliminaries}
In this section, we recall the Davis action on the exotic $15$-spheres $\Sigma_k^{15}$s, and the Brieskorn varieties with cohomogeneity one action. We refer to \cite{Baez} and \cite{Murakami} for the basics of the algebra of the Cayley numbers (i.e., the octonions) and the Lie group $\sfGt$.

\subsection{Shimada's exotic $15$-spheres $\Sigma_k^{15}$s, the embedded $13$- and $14$-spheres and the Davis action}

Consider the Cayley numbers $\Cay$ and let $u \mapsto \bar u$ be the standard conjugation. A real inner product on $\Cay$ is defined by $u \cdot v = 1/2(u \bar {v} + v \bar{u})$. Let $\set{e_0, e_1, \ldots, e_7}$ be an orthonormal basis of $\Cay$ over $\Real$ with $e_0 = 1$.  We follow the multiplications of elements in $\Cay$ given by \cite{Murakami}, for example, $e_1 e_2 = e_3$, $e_1e_4 = e_5$ and $e_1 e_7 = e_6$. Any $v \in \Cay$ has the following form
\[
v = v_0 e_0 + v_1 e_1 + \cdots + v_7 e_7.
\]  
Denote $\Re v = v_0$ the real part and $\Im v = v_1 e_1 + \ldots + v_7 e_7$ the imaginary part. We have 
\[
\bar v = v_0 e_0 - v_1 e_1 - \ldots - v_7 e_7
\] 
and 
\[
\abs{v}^2 = v_0^2 +v_1^2 + \cdots + v_7^2 = v\bar v.
\]
The unit $7$-sphere consists of all unit octonions:
\[
\sph^7 = \set{v \in \Cay : \abs{v}=1}.
\]
We write $\sph^8 = \Cay \sqcup_{\phi}\Cay$ as the union of two copies of $\Cay$ which are glued together along $\Cay -\set{0}$ via the following map
\begin{eqnarray}
\phi : \Cay -\set{0} & \To & \Cay - \set{0} \label{Cayleyphimap}\\
u & \mapsto & \phi(u) = \frac{u}{\abs{u}^2}. \nonumber
\end{eqnarray}
For any two integers $m$ and $n$, let $E_{m,n}$ be the manifold formed by gluing the two copies of $\Cay \times \sph^7$ via the following diffeomorphism on $\pset{\Cay -\set{0}} \times \sph^7$:
\begin{equation}\label{Phimnmap}
\Phi_{m,n}: (u, v) \mapsto (u', v')  = \left(\frac{u}{\abs{u}^2}, \frac{u^m}{\abs{u}^m} v \frac{u^n}{\abs{u}^n}\right).
\end{equation}
The natural projection $p_{m,n} : E_{m,n}\To\sph^8$ sends $(u,v)$ to $u$ and $(u',v')$ to $u'$. It gives $E_{m,n}$ the structure of an $\sph^7$-bundle over $\sph^8$ with the transition map $\Phi_{m,n}$. The total space $E_{m,n}$ is homeomorphic to $\sph^{15}$, if and only if, $m+n = \pm 1$; see \cite[Section 2]{Shimada}. 

Using the fact that $\sfGt$ is the automorphism group of $\Cay$, in \cite{Davis}, Davis observed that $\sfGt$ acts on $E_{m,n}$ as follows:
\[
g(u, v) = (g(u), g(v))
\]
and 
\[
g(u',v') = (g(u'), g(v')).
\]
From \cite[Remark 1.13]{Davis},  the $\sfGt$-manifolds $E_{m,n}$ and $E_{m',n'}$ are equivariantly diffeomorphic, whenever $(m,n) = \pm (m,n)$ or $\pm(n,m)$.  Furthermore, the bundles $E_{m,n}$ admit another $\SO(2)$ symmetry via M\"obius transformations that commutes with the $\sfGt$-action.  Write an element $\gamma \in \SO(2)$ as
\begin{equation}\label{eqn:gammaab}
\gamma = \gamma(a,b) =
\begin{pmatrix}
a & b \\
-b & a
\end{pmatrix}
\quad \text{and} \quad a^2 + b^2 = 1.
\end{equation}
In terms of the coordinate charts, the action on the sphere bundle $E_{m,n}$ is defined by
\begin{eqnarray}
\gamma \star u & = & (a u + b)(- b u + a)^{-1} \label{eqn:gammau} \\
\gamma \star u' & = & (-b + a u') (a + b u')^{-1} \nonumber
\end{eqnarray}
and 
\begin{eqnarray}
\gamma \star v & = & \frac{(-b u + a)^m v (- b u + a)^n}{\abs{-b u + a}^{m+n}} \label{eqn:gammav}\\
\gamma \star v' & = & \frac{(a+ b \bar u')^m v' (a+ b \bar u')^n}{\abs{a+ b \bar u'}^{m+n}}. \nonumber
\end{eqnarray}
The formulas above are compatible with the transition map $\Phi_{m,n}$. Davis showed the following
\begin{lem}[Davis]
The formulas (\ref{eqn:gammau}) and (\ref{eqn:gammav}) give a well-defined action of $\SO(2)$ on $E_{m,n}$. Furthermore the action is $\sfGt$-equivariant, and for any $v\in \Cay$(not necessarily unit) we have 
\[
\abs{\gamma \star v} = \abs{v} \quad \text{and} \quad \abs{\gamma\star v'} = \abs{v'}.
\]
\end{lem}

Suppose now that $m+n = 1$ and $k = m-n$. So $k$ is an odd number and 
\begin{equation}\label{eqn:mnk}
m = \frac{k+1}{2} \quad \text{and} \quad n = \frac{-k+1}{2}.
\end{equation}
We set $\Sigma_k^{15} = E_{m,n}$, and note that it is homeomorphic to the $15$-sphere. A Morse function on $\Sigma_k^{15}$ in \cite{Shimada} is given by 
\begin{equation}\label{eqn:Morsefunctionf1}
f_1(x) = \frac{\Re v}{\sqrt{1+\abs{u}^2}} = \frac{\Re (u' (v')^{-1})}{\sqrt{1+\abs{u' (v')^{-1}}^2}}.
\end{equation}
Note that $f_1$ has only two critical points as $(u,v) = (0, \pm 1)$. Set 
\begin{equation}\label{eqn:S14k}
\sph^{14}_k = f_1^{-1}(0) = \set{x \in \Sigma_k^{15} : \Re v  = \Re (u' (v')^{-1})= 0}
\end{equation}
and it is diffeomorphic to the standard $\sph^{14}$ for all $k$. Consider the following function on $\sph^{14}_k$:
\begin{equation}\label{eqn:Morsefunctionf2}
f_2(x) = \frac{\Re(uv)}{\sqrt{1+\abs{u}^2}} = \frac{\Re v'}{\sqrt{1+\abs{u'}^2}}.
\end{equation}
It is straightforward to verify that on $\sph^{14}_k$, the function $f_2$ has precisely two non-degenerate critical points as $(u',v') = (0, \pm 1)$.  It follows that
\begin{eqnarray}
\sph_k^{13} & = & f_2^{-1}(0) \cap \sph_k^{14} \nonumber\\
& = & \set{x \in \Sigma_k : \Re (uv) = \Re v = \Re v' = \Re (u' (v')^{-1}) = 0}  \subset \Sigma_k^{15} \label{eqn:S13k}
\end{eqnarray}
is diffeomorphic to the standard $13$-sphere for all $k$. Let 
\begin{eqnarray}\label{eqn:involutionT}
T  &: & E_{m,n} \To E_{m,n} \\
& & (u, v) \mapsto (u, -v) \quad \text{and} \quad (u', v') \mapsto (u', -v') \nonumber
\end{eqnarray}
be the antipodal map on the fiber $\sph^7$. The two spheres $\sph^{14}_k$ and $\sph^{13}_k$ are invariant under this involution $T$. Denote 
\[
P_k^{14} = \sph^{14}_k/T \quad \text{and} \quad P_k^{13} = \sph^{13}_k/T
\]
the quotient spaces.
  
\begin{rem}
Note that Milnor's exotic $7$-spheres $\Sigma^7$s are diffeomorphic to $3$-sphere bundles over the $4$-sphere. The involution $T$ on $\Sigma^{15}$s is the analogue of the natural involution on $\Sigma^7$s given by the antipodal map of the $3$-sphere fiber, see \cite{Milnor7sphere} and \cite{HirschMilnor}.
\end{rem}

In \cite{RajanWilhelm}, Wilhelm and the second named author observed that the Davis action of $\sfG = \SO(2) \times \sfGt$ on $\Sigma_k^{15}$ leaves both $\sph_k^{14}$ and $\sph_k^{13}$ invariant and commutes with the involution $T$.

\begin{lem}
The $\SO(2)\times \sfGt$ action on $\Sigma^{15}_k$ restricts to an action on the spheres $\sph^{14}_k$, $\sph^{13}_k$ and descends to the quotient spaces $P^{14}_k$, $P^{13}_k$.
\end{lem}
\begin{proof}
It is easy to see that the action commutes with the involution $T$. So it is sufficient to show that the defining conditions of $\sph^{13}_k$ and $\sph^{14}_k$ in $\Sigma_k^{15}$ are preserved by the $\SO(2)\times\sfGt$ action. In the following we give a proof for $\sph^{13}_k$, and the argument for $\sph^{14}_k$ is similar. 
 
Since $\sfGt$ is the automorphism group of $\Cay$, it is easy to see that the defining conditions are preserved. Next we consider the action by $\SO(2)$. Let $\gamma = \gamma(a,b)$ in equation (\ref{eqn:gammaab}). Note that $\Re (xy) = \Re(yx)$ for any $x, y \in \Cay$. We have
\begin{eqnarray*}
\Re\left(\gamma\star v\right) & = & \frac{1}{\abs{a- b u}}\Re\left\{(a- b u)^m v (a- b u)^n\right\} \\
& = & \frac{1}{\abs{a - b u}} \Re \left\{(a- bu)^{m+n} v \right\} \\
& = & \frac{1}{\abs{a- b u}}\left(a \Re v - b \Re(u v) \right) \\
& = & 0,
\end{eqnarray*}
and 
\begin{eqnarray*}
\Re\left((\gamma\star u) (\gamma \star v)\right) & = & \frac{1}{\abs{a- b u}}\Re \left\{(a u + b)(a- b u)^{-1}(a- bu)^m v (a- bu)^{n}\right\} \\
& = & \frac{1}{\abs{a- b u}}\Re (au + b)v \\
& = & 0.
\end{eqnarray*}

For the coordinates $(u', v')$, since $u' (v')^{-1} = u' \bar v' /\abs{v'}^2$ and $\Re \left(u' (v')^{-1} \right)= 0$; it follows that $\Re \left(\bar u' v'\right) = 0$. Similar to the case of $(u,v)$, we have
\begin{eqnarray*}
\Re (\gamma \star v') & = & \frac{1}{\abs{a+ b \bar u '}} \Re \left\{(a+ b \bar u')^m v' (a+ b \bar u')^n\right\} \\
& = & \frac{1}{\abs{a+ b \bar u'}} \Re\left\{ (a + b \bar u') v'\right\} \\
& = & 0
\end{eqnarray*}
and
\begin{eqnarray*}
\Re\left((\gamma \star u')(\gamma \star v')^{-1}\right) & = & \abs{a+b \bar u'}\Re \left\{(-b + au')(a + b u')^{-1}(a+ b\bar u')^{-n}(v')^{-1}(a+ b\bar u')^{-m}\right\} \\
& = & \abs{a+ b\bar u'} \Re\left\{(-b + au')(a + b u')^{-1}(a + b \bar u')^{-1} (v')^{-1}\right\} \\
& = & \abs{a+ b \bar u'}\left(a^2 + b^2 \abs{u'}^2 + ab (u' + \bar u')\right)\Re \set{(-b + au') (v')^{-1}} \\
& = & 0.
\end{eqnarray*}
This shows that $\sph^{13}_k$ is invariant under the $\SO(2)$ action, which finishes the proof.
\end{proof}

\begin{rem}
In \cite{RajanWilhelm}, following the Hirsch-Milnor argument in \cite{HirschMilnor}, they also showed that $P_k^{14}$ and $P_k^{13}$ are homotopy equivalent to the standard $\RP^{14}$ and $\RP^{13}$ for all $k$; and $P^{14}_k$ is not diffeomorphic to $\RP^{14}$, when $k \equiv 3, 5\mod 8$.
\end{rem}

\subsection{Brieskorn varieties, Kervaire spheres and homotopy projective spaces}

For any integers $n\geq 3$ and $d\geq 1$, the Brieskorn variety $M^{2n-1}_d$ is the smooth $(2n-1)$-dimensional submanifold of $\Cpx^{n+1}$, defined by the equations
\[
\left\{
\begin{array}{rcl}
z_0^{d} + z_1^{2} + \cdots + z_n^{2} & = & 0 \\
\abs{z_0}^2 + \abs{z_1}^2 + \cdots + \abs{z_n}^2 & = & 1. 
\end{array}
\right.
\]
When $d=1$, $M_1^{2n-1}$ is diffeomorphic to the standard sphere $\sph^{2n-1}$; and when $d=2$, $M_2^{2n-1}$ is diffeomorphic to the unit tangent bundle of $\sph^{n}$.

\begin{thm}[Brieskorn]
Suppose $n\geq 3$ and $d\geq 2$. The manifold $M^{2n-1}_d$ is homeomorphic to the standard sphere $\sph^{2n-1}$, if and only if, both $n$ and $d$ are odd numbers. Assume that $n$ and $d$ are odd numbers, it is the Kervaire sphere, if and only if, $d\equiv \pm 3 \mod8$. 
\end{thm}
\begin{rem}
The Kervaire sphere is known to be exotic if $n\equiv 1 \mod 4$.
\end{rem}

Denote $I$ the following involution on $M^{2n-1}_d$: 
\[
(z_0, z_1, \ldots, z_n)\mapsto (z_0, - z_1, \ldots, -z_n).
\] 
Clearly it is fixed-point free. Atiyah and Bott showed the following result, see also \cite[Corollary 4.2]{Giffen}. 
\begin{thm}[{{\cite[Theorem 9.8]{AtiyahBott}}}]
If the involution $I$ on the topological spheres $M^{4m-3}_d$ and $M^{4m-3}_k$ are isomorphic, then 
\[
d \equiv \pm k \mod 2^{2m}. 
\]
In particular the involution $I$ acting on $M^{4m-3}_3 = \sph^{4m-3}$ is not isomorphic to the standard antipodal map whenever $m \geq 2$.
\end{thm}
\begin{cor}
There are $64$ smoothly distinct real projective spaces $M^{13}_k/I$ with $k = 1, 3, \ldots, 127$. 
\end{cor}

The group $\tilde{\sfG} = \SO(2) \times \SO(n)$ acts on $M_d^{2n-1}$ by
\[
\left(e^{i \theta}, A\right)\left(z_0, Z\right) = \left(e^{2i \theta}z_0, e^{-i d \theta} A Z\right), \quad\text{for } (z_0, Z) \in \Cpx \oplus \Cpx^n.
\]
Note that our convention is different from the one in \cite{GVWZ}, as we have $e^{-id\theta}$ for the action of $e^{i\theta}$ on $Z = \pset{z_1, \ldots, z_n}^t$. The norm $\abs{z_0}$ is invariant under this action, and two points belong to the same orbit if and only if they have the same value of $\abs{z_0}$. Let $t_0$ be the unique positive solution of $t_0^d + t_0^2 = 1$, and then we have $0\leq \abs{z_0}\leq t_0$. It follows that the orbit space is $[0, t_0]$. The orbit types and isotropy subgroups of this action have been well-studied, see for example, \cite{HsiangHsiang}, \cite{BackHsiang} and \cite{GVWZ}. 

In our case, we assume that $d$ is odd. When $n=7$, the embedding $\sfGt \subset \SO(7)$ induces the action of $\sfG = \SO(2) \times \sfGt$ on $M^{13}_d$. To describe the isotropy subgroups of the $\sfG$-action we introduce the following subgroups in $\sfGt$:  
\begin{itemize}
\item Denote $\sfO(6)$, the subgroup in $\SO(7)$ that maps $e_1$ to $\pm e_1$, $\SO(6)$ the subgroup that fixes $e_1$, and $\SU(3) = \SO(6) \cap \sfGt$.

\item The other subgroup in $\sfGt$ that fixes $e_3$ is denoted by $\SU(3)_3$, and the complex structure on $\Cpx^3 = \Span_\Real \set{e_1, e_2, e_4, e_7, e_6, e_5}$ is given by the left multiplication of $e_3$. Note that
\[
\left(\SO(2)\times\SO(5)\right) \cap \sfGt = \U(2) \subset \SU(3)_3
\]
where $\SO(2)\times \SO(5) \subset \SO(7)$ has the block-diagonal form, and the embedding $\U(2) \subset \SU(3)_3$ is given by $h \mapsto \diag\set{(\det h)^{-1}, h}$. To see this, take $A = \diag\set{A_1, A_2}\in \left(\SO(2) \times \SO(5)\right)\cap \sfGt$ with 
\[
A_1 = 
\begin{pmatrix}
\cos t & \sin t \\
-\sin t & \cos t
\end{pmatrix}
\]
for some $t$. Since $e_3 =e_1 e_2$, we have 
\begin{eqnarray*}
A(e_3) & = & A(e_1)A(e_2)\\
& = &\left(e_1 \cos t + e_2 \sin t\right) \left(- e_1\sin t + e_2\cos t\right)\\
& = & e_3
\end{eqnarray*}
and thus $A \in \SU(3)_3$. Using the complex structure of $\SU(3)_3$, $A_1$ acts on $\Cpx = \Span_\Real\set{e_1, e_2}$ by $e^{it}$, and $A_2$ acts invariantly on $\Cpx^2 = \Span_\Real\set{e_4, e_7, e_6, e_5}$. So the element $A$ embeds diagonally in $\SU(3)_3$ with $(1,1)$-entry $e^{it}$. 

\item The common subgroup $\SU(2) = \SU(3) \cap \SU(3)_3$ and it is also given by $\SU(2) = \SO(4) \cap \sfGt$ where $\SO(4) \subset \SO(7)$ as $A \mapsto \diag\set{I_3, A}$ and $I_3$ is the identity matrix. 
\end{itemize}

Since $\sfGt$ acts transitively on $\sph^6 = \set{v \in \Cay : \Re v = 0 \text{ and } \abs{v} = 1}$ with $\SU(3)$ and $\SU(3)_3$ as isotropy subgroups at $e_1$ and $e_3$ respectively, these two groups are conjugate by an element in $\sfGt$. 

We follow the notions in \cite{GVWZ} to determine the isotropy subgroups. Denote $\Bm$ the singular orbit with $\abs{z_0} = 0$, and choose $p_- = (0,1,i, 0,\ldots, 0) \in \Bm$ with isotropy subgroup $\Km$. We also denote $\Bp$ the singular orbit with $\abs{z_0} = t_0$, and choose $p_+ = (t_0, i \sqrt{t_0^d}, 0, \ldots, 0)$ with isotropy subgroup $\Kp$. Note that $\Bm$ and $\Bp$ have codimensions $2$ and $n-1 = 6$ respectively. Let $c(t)$ be a normal minimal geodesic connecting $p_- = c(0)$ and $p_+ = c(L)$. The isotropy subgroup at $c(t)$($0< t< L$) stays unchanged that is the principal isotropy subgroup $\sfH$. We have

\begin{thm}\label{thm:groupsM13d}
The cohomogeneity one action of $\sfG = \SO(2) \times \sfGt$ on $M_d^{13}$ with $d$ odd has the following isotropy subgroups:
\begin{enumerate}
\item The principal isotropy subgroup is
\[
\sfH = \Zeit_2 \cdot \SU(2) = \left(\eps, \diag\set{\eps, \eps, 1, A}\right) 
\]
where $\eps = \pm 1$ and $A$ is a $4\times 4$-matrix. 
\item At $p_-$, the isotropy subgroup is 
\[
\Km = \SO(2) \SU(2) = \left(e^{i \theta}, \diag\set{\begin{pmatrix}\cos d\theta & \sin d\theta \\ - \sin d\theta & \cos d\theta \end{pmatrix}, 1, A}\right)
\]
where $A$ is a $4\times 4$-matrix.
\item At $p_+$, the isotropy subgroup is 
\[
\Kp = \sfO(6)\cap \sfGt = \left(\det B, \diag\set{\det B, B}\right)
\]
where $B \in \sfO(6)\cap \sfGt$.
\end{enumerate}
\end{thm}

\begin{rem}\label{rem:M13dgroups}
Denote $j$, the complex structure given by the left multiplication of $e_3$. For the group $\sfH$, we have $\diag\set{\eps, \eps, 1, A} \in \pset{\SO(2) \times \SO(5)}\cap \sfGt$ and $A \in \U(2) \subset \SU(3)_3$ with $\det A = \eps$. For the group $\Km$, we have 
\[
 \diag\set{\begin{pmatrix}\cos d\theta & \sin d\theta \\ - \sin d\theta & \cos d\theta \end{pmatrix}, 1, A}\in \pset{\SO(2) \times \SO(5) }\cap \sfGt
\]
and $A\in \U(2) \subset \SU(3)_3$ with $\det A = e^{- j d \theta}$. 
\end{rem}

\begin{rem}\label{rem:M13deven}
If $d$ is an even integer, then the isotropy subgroup $\Km$ is the same as in the case $d$ odd. The other two isotropy subgroups are 
\begin{eqnarray*}
\sfH = \Zeit_2 \times \SU(2) = \pset{\eps, \diag\set{I_3, A}} \\
\Kp = \Zeit_2 \times \SU(3) = \pset{\eps, \diag\set{1, B}}
\end{eqnarray*}
where $\eps = \pm 1$, $A \in \SO(4) \cap \sfGt = \SU(2)$ and $B \in \SO(6) \cap \sfGt = \SU(3)$.
\end{rem}

Clearly the ${\sfG}$-action commutes with the involution $I$ and hence induces an action on $N_d^{13} = M_d^{13}/I$. Write $[z_0, z_1, \ldots, z_7]\in N^{13}_d$, the equivalent class under the involution $I$. 

\begin{cor}\label{cor:groupsN13d}
The cohomogeneity one action of $\sfG = \SO(2) \times \sfGt$ on $N^{13}_d = M^{13}_d/I$ with $d$ odd, has the following isotropy subgroups.
\begin{enumerate}
\item The principal isotropy subgroup is 
\[
\bar \sfH = \Zeit_2 \times \left(\Zeit_2 \cdot \SU(2)\right)  =\left(\eps_1, \diag\set{ \eps_2, \eps_2, 1, A}\right)
\]
where $\eps_{1,2} = \pm 1$ and $A$ is a $4\times 4$-matrix.
\item The singular isotropy subgroup at $[0,1,i, 0, \ldots, 0]$ is
\[
\barKm = \Zeit_2 \cdot\SO(2) \SU(2) = \left(e^{i\theta}, \diag\set{\eps \begin{pmatrix}\cos d\theta & \sin d\theta \\ - \sin d\theta & \cos d\theta \end{pmatrix}, 1, A}\right)
\]
where $\eps = \pm 1$ and $A$ is a $4\times 4$-matrix.
\item The singular isotropy subgroup at $[t_0, i \sqrt{t_0^d}, 0, \ldots, 0]$ is 
\[
\barKp = \Zeit_2 \times \left(\sfO(6) \cap \sfGt\right) = \left(\eps, \diag\set{\det B, B}\right)
\] 
where $\eps = \pm 1$ and $B \in \sfO(6) \cap \sfGt$. 
\end{enumerate}
\end{cor}
\begin{rem}
Similar to Remark \ref{rem:M13dgroups}, for the group $\bar\sfH$ we have $A\in \U(2) \subset \SU(3)_3$ with $\det A = \eps_2$, and for the group $\barKm$ we have $A \in \U(2) \subset \SU(3)_3$ with $\det A = \eps e^{- jd\theta}$. 
\end{rem}

\medskip
\section{The cohomogeneity one actions of $\sfG =\SO(2) \times \sfGt$ on $\sph^{13}_k$ and $P^{13}_k$}

In this section we determine the cohomogeneity one action of $\sfG$ on $\sph^{13}_k$ and $P^{13}_k$, see Theorem \ref{thm:isotropyS13} and Corollary \ref{cor:isotropyP13}. Then we prove Theorem \ref{thm:introBrieskornvarieties} in the Introduction. At the end of this section, we determine the Weyl group of the cohomogeneity one action on $M_k^{13}$, see Proposition \ref{prop:Weylgroupelements}. 

Throughout this section, we assume that $k$ is an odd integer. For the basics of cohomogeneity one manifolds, we refer to \cite[Section 1]{GroveWilkingZiller}.

\smallskip

Since the actions of $\SO(2)$ and $\sfGt$ commute, we determine the orbit space $\calB$ of $\sph^{13}_k$ under the $\sfGt$ action, and then consider the $\SO(2)$-action on $\calB$.

\begin{prop}\label{prop:G2orbitS13}
The orbit space of $\sph^{13}_k$ under the $\sfGt$-action is 
\[
\calB^2 =\calB_1 \sqcup_{\Phi} \calB_2
\]
with $\calB_1 \cong \calB_2 \cong \Real \times [0, \infty)$, where the two charts are determined as follows:
\begin{enumerate}
\item the point $[x_1 + x_2 e_3,e_1]$ in $\calB_1$ is identified with the $\sfGt$-orbit at $(x_1 + x_2 e_3, e_1)$ in the chart with coordinates $(u,v)$;
\item the point $[x_1' + x_2' e_3,e_1]$ in $\calB_2$ is identified with the $\sfGt$-orbit at $(x_1' + x_2' e_3, e_1)$ in the chart with coordinates $(u',v')$,
\end{enumerate}
and the gluing map $\Phi : \calB_1 \backslash\set{0} \To \calB_2 \backslash\set{0}$ is given by
\[
\Phi\left([x,e_1]\right) = \left[x/\abs{x}^{2}, e_1\right] \quad \text{for any }  x= x_1 + x_2 e_3 \ne 0.
\]
\end{prop}
\begin{proof}
On the chart with coordinates $(u,v)$ we have $\Re v = 0$ and $\abs v =1$, i.e., $v\in \sph^6 \subset \Im \Cay$. Write $u = u_0 + u_1$ with $u_1 \in \Im \Cay$. Then the condition $\Re(u v) = 0$ is equivalent to $\bk{u_1, v} = 0$. Since $\sfGt$ acts transitively on $\sph^6$, there exists some $\sigma_1 \in \sfGt$ such that $e_1 = \sigma_1(v)$, and then $\sigma_1(u) = u_0 + \sigma_1(u_1)$ with $\sigma_1(u_1) \in \Im \Cay$. The left multiplication of $e_1$ induces a complex structure on the space $\Cpx^3 = \Span_{\Real}\set{e_2, \cdots, e_7}$. The isotropy subgroup at $e_1 \in \sph^6$ is $\SU(3)$. Note that we also have $\bk{e_1, \sigma_1(u_1)} = 0$. Since $\SU(3)$ acts transitively on $\sph^5 \subset \Cpx^3$, there is $\sigma_2 \in \SU(3) \subset \sfGt$ such that $\sigma_2(\sigma_1(u_1)) = \abs{u_1} e_3$. Let $\sigma = \sigma_2 \sigma_1 \in \sfGt$, then we have $\sigma(u, v) = (u_0 + \abs{\Im u} e_3, e_1)$.   

Next we consider the chart with coordinates $(u', v')$. First, we have $v' \in \sph^6 \subset \Im \Cay$. Write $u' = u'_0 + u'_1$ with $u'_1 \in \Im \Cay$. Then the condition $\Re(u'(v')^{-1}) = 0$ is equivalent to $\Re (\bar u' v') = 0$, i.e., $\bk{u'_1, v'} = 0$. Similar to the argument for $(u,v)$, there is a $\tau_1 \in \sfGt$ such that $e_1 = \tau_1(v')$ and $\bk{e_1, \tau_1(u_1')} = 0$. Then there is a $\tau_2 \in \SU(3)$, the isotropy subgroup of $e_1$ in $\sfGt$, such that $\tau_2(\tau_1(u_1')) = \abs{u_1'}e_3$. It follows that $\tau(u',v') = (u'_0 + \abs{\Im u'}e_3, e_1)$ with $\tau = \tau_2\tau_1 \in \sfGt$.

Now we consider the transition map $\Phi_{m,n}$. Let $(u, v) = \sigma (x_1 + x_2 e_3, e_1)$ with $(x_1, x_2)\in \Real \times [0, \infty)$, i.e., $u=\sigma(x_1+ x_2e_3)$ and $v = \sigma(e_1)$. Write 
\[
x_1 + x_2 e_3 = r \left(\cos \theta + \sin \theta e_3\right)
\]
for some $\theta \in [0, \pi]$. Then the image $(u', v') = \Phi_{m,n}(u,v)$ is given by 
\begin{eqnarray*}
u' & = & \sigma \left(\frac{x_1+ x_2 e_3}{x_1^2 + x_2^2}\right) = \sigma \left(\frac{\cos\theta + \sin \theta e_3}{r}\right) \\
v' & = & \sigma \left(\frac{(x_1+x_2 e_3)^m e_1 (x_1+x_2 e_3)^n}{\abs{x_1 + x_2e_3}}\right) = \sigma \left\{\left(\cos (k \theta) + \sin (k \theta) e_3\right) e_1\right\},
\end{eqnarray*}
i.e., $(u',v')$ is in the orbit of $(r^{-1}(\cos \theta + \sin \theta e_3), (\cos (k\theta) + \sin (k\theta)e_3)e_1)$. Since all orbits have a point with $(y_1 + y_2 e_3, e_1)$ with $y_2 \geq 0$, it follows that there exists a $\tau \in \sfGt$ such that 
\begin{eqnarray*}
\frac{1}{r}\left(\cos \theta + \sin \theta e_3\right) & = & \tau(y_1 + y_2 e_3) \\
\cos(k\theta) e_1 + \sin(k\theta) e_2 & = & \tau(e_1).
\end{eqnarray*}
In fact we may choose $\tau$ such that it fixes $e_3$, and rotates in $\set{e_1, e_2}$-plane by the second equation above and the space spanned by $\set{e_4, \ldots, e_7}$. Such $\tau$ exists in another copy of $\SU(3)$, which is the isotropy subgroup of $e_3$. Denote $[u,v]$ and $[u',v']$, the $\sfGt$-orbits in coordinate charts $(u,v)$ and $(u',v')$ respectively. In a summary, under the transition map $\Phi_{m,n}$, we have
\[
\Phi_{m,n} \left(\left[r (\cos \theta + \sin \theta e_3), e_1\right]\right) = \left[\frac{1}{r}(\cos \theta + \sin \theta e_3), e_1\right]
\]
which defines the map $\Phi$. This finishes the proof.
\end{proof}

Next, we consider the $\SO(2)$-action on the orbit space $\calB^2$. Recall 
\begin{equation}\label{eqn:gammaab2}
\gamma = \gamma(a,b) = 
\begin{pmatrix}
a & b \\
- b & a
\end{pmatrix}
\end{equation}
with $a^2 + b^2 = 1$. 

\begin{prop}\label{prop:gammaactionX}
Let $\gamma$ be an element in $\SO(2)$ as in (\ref{eqn:gammaab2}). Then $\gamma$ acts on the $\sfGt$-orbit space $\calB^2 = \calB_1 \sqcup_\Phi \calB_2$ as follows.
\begin{enumerate}[(1)]
\item If $b = 0$, then we have 
\begin{eqnarray*}
\gamma \star (u, v) & = & (u, \mathrm{sgn}(a) v) \\
\gamma \star (u', v') & = & (u', \mathrm{sgn}(a) v')
\end{eqnarray*}
on the $(u,v)$- and $(u', v')$-coordinate charts. 

\item If $b \ne 0$, then we have
\begin{eqnarray*}
\gamma \star [u_1 + u_2 e_3,e_1] & = & \left[-\frac{a}{b} + \frac{a- b u_1}{b \left((a- b u_1)^2 + b^2 u_2^2\right)} + \frac{u_2}{(a- b u_1)^2 + b^2 u_2^2} e_3,e_1\right] \\
\gamma \star [u_1' + u_2' e_3,e_1] & = & \left[\frac{a}{b} - \frac{a+ b u_1'}{b \left((a+ b u_1')^2 + b^2 (u'_2)^2\right)} + \frac{u_2'}{(a+ b u_1')^2 + b^2 (u_2')^2} e_3, e_1\right]
\end{eqnarray*}
where $[u_1 + u_2 e_3, e_1] \in \calB_1$ and $[u_1' + u_2' e_3, e_1]\in \calB_2$.
\end{enumerate}
\end{prop}

\begin{proof}
Take $(u, v)\in \sph^{13}_k$ through the orbit $[u_1 + u_2 e_3, e_1] \in \calB_1$ and write $a - b\bar u = r(\cos \theta + \sin \theta e_3)$, i.e., 
\[
\left\{
\begin{array}{rcl}
a - b u_1 & = & r\cos \theta \\
b u_2 & = & r \sin \theta. 
\end{array}
\right.
\]
then
\begin{equation}\label{eqn:u1u2rtheta}
\left\{
\begin{array}{rcl}
u_1 & = & \dfrac{1}{b}(a-r\cos \theta) \\
& & \\
u_2 & = & \dfrac{r}{b} \sin \theta. 
\end{array}
\right.
\end{equation}

\begin{claim} We have 
\begin{eqnarray*}
\gamma\star u & = & - \frac{a}{b} + \frac{1}{r b}\left(\cos \theta + \sin \theta e_3 \right) \\
\gamma \star v & = & e_1 \left(\cos(k \theta)  + \sin (k \theta) e_3\right). 
\end{eqnarray*}
\end{claim}
It follows from a straightforward computation. We have
\begin{eqnarray*}
\gamma\star u & = & (a u + b)(a- bu)^{-1} \\
& = & (a u_1 + b + a u_2 e_3)\frac{a-b \bar u}{\abs{a- b u}^2} \\
& = & \frac{(a^2 + b^2)\cos\theta - a r + (a^2 + b^2)\sin\theta e_3}{rb} \\
& = & \frac{-r a + \cos\theta + \sin \theta e_3}{rb}.
\end{eqnarray*}
This gives the first formula. Then we have
\begin{eqnarray*}
\gamma\star v & = & \frac{(a- b u)^{m} e_1 (a- bu)^n}{\abs{a- bu}} \\
& = & e_1 \frac{(a- b\bar u)^m(a- bu)^{1-m}}{r} \\
& = & e_1 \frac{(a- b\bar u)^m(a - b\bar u )^{m-1}}{r^{2m-1}}  \\
& = & e_1 \left(\cos (2m-1)\theta + e_3 \sin (2m-1) \theta\right).
\end{eqnarray*}
This gives the second formula, as $2m-1 = k$. This finishes the proof of the claim. 

Next we derive the action of $\gamma$ on chart with coordinates $(u',v')$. Take $(u', v')\in \sph^{13}_k$, through the orbit $[u'_1 + u'_2 e_3, e_1] \in \calB_2$ with $u'_2 \geq 0$. Write $a + b\bar u' = r(\cos t +\sin t e_3)$, i.e.,
\[
\left\{
\begin{array}{rcl}
a + b u'_1 & = & r\cos t \\
- bu'_2 & = & r \sin t.
\end{array}
\right.
\]
A straightforward computation shows the following:
\begin{eqnarray*}
\gamma\star u' & = & \frac{a}{b} - \frac{1}{r b}\left(\cos t + e_3 \sin t\right) \\
\gamma\star v' & = & e_1 \left(\cos (kt) - \sin (kt) e_3\right).
\end{eqnarray*}
From a similar argument in Proposition \ref{prop:G2orbitS13}, both $(\gamma \star u, \gamma\star v)$ and $(\gamma \star u, e_1)$ are in the same $\sfGt$-orbit. This also holds for $(u', v')$ and thus we finish the proof. 
\end{proof}

\begin{rem}
(a) One can see that the action of $\gamma$ on $\calB = \calB_1\sqcup_\Phi \calB_2$ is compatible with the map $\Phi$. Restrict $\Phi$ to the first component. Take $u = u_1 + u_2 e_3$ and $u' = \Phi(u) = u_1' + u_2' e_3$ with 
\begin{eqnarray*}
u_1' & = & \frac{u_1}{u_1^2 + u_2^2} \\
u_2' & = & \frac{u_2}{u_1^2 + u_2^2}.
\end{eqnarray*} 
Then a direct calculation shows that $\Phi(\gamma \star u) = \gamma \star u'$.

(b) Restricted to the $u$ and $u'$-component, the action of $\gamma$ is the M\"{o}bius transformation of the upper half plane with the identification 
\[
u_1 + u_2 e_3 \sim u_1 + i u_2.
\]
The unique fixed point is $e_3$ with $(u_1, u_2) = (0, 1)$.  The action of $\SO(2)$ is by isometries with respect to the hyperbolic metric
\[
ds^2 = \frac{du_1^2 + d u_2^2}{u_2^2},
\]
so that we can identify the orbit spaces as the line segment $\set{u_2 e_3 : 0 \leq u_2 \leq 1}$.
\end{rem}

\begin{thm}\label{thm:isotropyS13}
The cohomogeneity one action of $\sfG = \SO(2) \times \sfGt$ on $\sph^{13}_k$ has the following isotropy subgroups:
\begin{enumerate}[(1)]
\item At $(e_3, e_1)$ in the $(u,v)$-coordinate chart, the isotropy subgroup is 
\begin{equation*}
K = \SO(2) \SU(2) = \left(e^{i \theta}, \diag\set{\begin{pmatrix}\cos k\theta & \sin k\theta \\ - \sin k\theta & \cos k\theta \end{pmatrix}, 1, A}\right)
\]
where $A$ is a $4\times 4$-matrix.
\item At $(u_1, e_1)$ in the $(u,v)$-coordinate chart with $u_1 \in \Real$, or $(0,e_1)$ in the $(u',v')$-coordinate chart, the isotropy subgroup is 
\begin{equation*}
L = \sfO(6) \cap \sfGt = \left(\det B, \left(\begin{smallmatrix} \det B & 0 \\ 0 & B\end{smallmatrix}\right)\right) 
\end{equation*}
where $B \in \sfO(6) \cap \sfGt$.
\item At $(u_1+ u_2 e_3, e_1)$ in the $(u,v)$-coordinate chart with $(u_1, u_2) \in \Real \times (0, \infty) - (0, 1)$, the isotropy subgroup is 
\begin{equation*}
H = \Zeit_2 \cdot \SU(2) = \left(\eps, \diag\set{\eps, \eps, 1, A}\right)
\end{equation*}
where $\eps = \pm 1$ and $A$ is a $4\times 4$-matrix. 
\end{enumerate}
 \end{thm}
 
 \begin{proof}
 Suppose $q = g(p)$ for some $g \in \sfGt$. Then the isotropy subgroups have the following relation:
\[
\sfG_q = \set{(\gamma, h) \in \SO(2) \times \sfGt : (\gamma, g^{-1}h g ) \in \sfG_p},
\]
i.e., $g^{-1} \sfG_q g = \sfG_p$. So it is sufficient to just consider the isotropy subgroups on $\mathcal B^2$. From Proposition \ref{prop:G2orbitS13}, we only need to consider the $(u,v)$-coordinate chart, and the point $(0, e_1)$ in the $(u',v')$-coordinate chart. 
 
We first consider the isotropy subgroup at $(u, v) = (u_1 + u_2 e_3, e_1) \in \sph^{13}_k$. Choose an element $(\gamma^{-1}, h)$, with $\gamma = \gamma(a,b) \in \SO(2)$ given by equation (\ref{eqn:gammaab2}) and $h \in \sfGt$. Suppose that $(\gamma^{-1}, h) \in \sfG_{(u,v)}$, we have 
\[
h(u, v) = \gamma\star (u,v).
\] 
In the first case we assume that the isotropy subgroup contains an element $(\gamma^{-1}, h)$ with $b \ne 0$. Write $(u_1, u_2)$ in terms of $(r,\theta)$ as in equations (\ref{eqn:u1u2rtheta}). Following Proposition \ref{prop:gammaactionX}, we have
\begin{eqnarray*}
- \frac{a}{b} + \frac{1}{rb}(\cos\theta + \sin \theta e_3) & = & \frac{a}{b} - \frac{r}{b}\cos \theta +h(e_3) \frac{r}{b} \sin \theta \\
e_1\cos(k \theta) - e_2 \sin (k \theta) & = & h(e_1).
\end{eqnarray*}
Since $\Re h(e_3) = 0$, these two equations above are equivalent to the following equations: 
\begin{eqnarray*}
2a & = & \left( r + \frac 1 r\right)\cos \theta \\
(r \sin \theta) h (e_3) & = & \frac{\sin \theta}{r} e_3 \\
h(e_1) & = & e_1 \cos(k\theta) - e_2 \sin (k \theta).
\end{eqnarray*}
If $\sin \theta = 0$, then $\cos \theta = \pm 1$. From the first equation above we have, either $a \geq 1$ or $a \leq -1$. In either case, we have $b = 0$ that contradicts our assumption that $b \ne 0$. So we have $\sin \theta \ne 0$, and thus the second equation implies that $h(e_3) = r^{-2} e_3$.  It follows that $r = 1$ and $a = \cos \theta$ from the first equation. From equations (\ref{eqn:u1u2rtheta}) we have $u_1 = 0$, $u_2 = 1$ and $b = \sin \theta$. In this case $h$ is the rotation in the plane $\Span_\Real \set{e_1, e_2}$ while fixing $e_3$.  The left multiplication of $e_3$ defines a complex structure on the vector space $\Span_\Real\set{e_1, e_2, e_4, \ldots, e_7}$ and
\[
h
\begin{pmatrix}
e_1 \\
e_2
\end{pmatrix} = 
\begin{pmatrix}
\cos k \theta & - \sin k \theta \\
\sin k \theta & \cos k \theta
\end{pmatrix}
\begin{pmatrix}
e_1 \\
e_2
\end{pmatrix}.
\]
So we have $(u, v) = (e_3, e_1)$, $\gamma = R(\theta)$ and $h|_{\set{e_1, e_2}}= R(-k \theta)$. It follows that $(\gamma^{-1}, h)\in K$ in Case (1). 

In the second case we assume that $b = 0$. Suppose that $a = 1$, then we have $\gamma \star (u,v) = (u,v)$. It follows that $h(u, v) = (u,v)$, i.e.,
\begin{eqnarray*}
h(u_1 + u_2 e_3) & = & u_1 + u_2 e_3 \\
h(e_1) & = & e_1.
\end{eqnarray*}
It follows that $h\in \SU(3)$ if $u_2 = 0$. If $u_2 \ne 0$, then we have $h(e_3) = e_3$, and so $h \in \SU(2)$. Now suppose that $a = -1$ and we have $\gamma \star (u, v) = (u, -v)$. It follows that $h(u, v) = (u, -v)$, i.e., 
\begin{eqnarray*}
h(u_1 + u_2 e_3) & = & u_1 + u_2 e_3 \\
h(e_1) & = & - e_1.
\end{eqnarray*}  
If $u_2 = 0$, then we have $h(e_1) = - e_1$. If $u_2 \ne 0$, then we have $h(e_3) = e_3$ and $h(e_1) = -e_1$. It follows that the isotropy subgroup at $(u_1, e_1)$ is $L$ as in Case (2), and the identity component is 
\[
L_0 = \set{(1, A) : A \in \SU(3) \subset \sfGt}.
\]
The isotropy subgroup at $(u_1 + u_2 e_3, e_1)$ with $u_2 >0$ and $(u_1, u_2) \ne (0,1)$ is $H$ as in Case (3).

Next we consider the isotropy subgroup at $(u',v') = (0, e_1)$. Suppose that $(\gamma^{-1}, h) \in \sfG_{(0, e_1)}$ with $\gamma$ being given by (\ref{eqn:gammaab2}). If $b \ne 0$, then from Proposition \ref{prop:gammaactionX}, we have 
\[
0 = \frac{a}{b} - \frac{1}{ab}
\]
i.e., $a^2 = 1$ and thus $b = 0$. So we have $b = 0$ and $\gamma\star (0,e_1) = (0, \mathrm{sgn}(a) e_1)$. It follows that $h(e_1) = \mathrm{sgn}(a) e_1$. So we have $(\gamma^{-1}, h) \in L$ as in Case (2). This finishes the proof.
\end{proof}

\begin{cor}\label{cor:isotropyP13}
The cohomogeneity one action of $\sfG = \SO(2) \times \sfGt$ on $P^{13}_k$ has the following isotropy subgroups
\begin{eqnarray*}
\bar K & = & \Zeit_2 \cdot\SO(2) \SU(2) = \left(e^{i\theta}, \diag\set{\eps \begin{pmatrix}\cos k\theta & \sin k\theta \\ - \sin k\theta & \cos k\theta \end{pmatrix}, 1, A}\right) \\
& & \text{where } \eps = \pm 1 \text{ and } A \text{ is a } 4\times 4\text{-matrix}, \\
\bar L & = & \Zeit_2 \times \left(\sfO(6) \cap \sfGt\right) = \left(\eps, \diag\set{\det B, B}\right) \\
& & \text{where } \eps = \pm 1 \text{ and } B \in \sfO(6) \cap \sfGt, \\
\bar H & = & \Zeit_2 \times \left(\Zeit_2 \cdot \SU(2)\right)  =\left(\eps_1, \diag\set{ \eps_2, \eps_2, 1, A}\right) \\
& & \text{where } \eps_{1,2} = \pm 1 \text{ and } A \text{ is a } 4\times 4\text{-matrix}.
\end{eqnarray*}
\end{cor}

Now we show the equivariant diffeomorphisms between $\sph^{13}_k$ and $M^{13}_k$, and between $P^{13}_k$ and $N_k^{13}$.
\begin{proof}[Proof of Theorem \ref{thm:introBrieskornvarieties}]
From the general structure result, see for example \cite[Section 1]{GroveWilkingZiller}, two cohomogeneity one manifolds with the same isotropy subgroups are equivariantly diffeomorphic. In our case, let $\mathbb D^2$ and $\mathbb D^6$ be disks with $\partial \mathbb D^2 = \sph^1 = \Km/\sfH$ and $\partial \mathbb D^6 = \sph^5 = \Kp/\sfH$ with $\mathsf{K}^{\pm}$ and $\sfH$ being given in Theorem \ref{thm:groupsM13d}. Then $M^{13}_k$ is equivariantly diffeomorphic to the union of the two disk bundles glued together along the boundary $\sfG/\sfH$:
\[
B^{13} = \sfG \times_{\Km} \mathbb D^2 \cup_{\sfG/\sfH} \sfG \times_{\Kp} \mathbb D^6.
\]
From Theorem \ref{thm:isotropyS13}, the sphere $\sph^{13}_k$ is also equivariantly diffeomorphic to the $B^{13}$ above. It follows that $\sph^{13}_k$ is equivariantly diffeomorphic to $M^{13}_k$. The equivariant diffeomorphism between $P^{13}_k$ and $N_k^{13}$ follows from a similar argument and Corollaries \ref{cor:groupsN13d}, \ref{cor:isotropyP13}. This finishes the proof. 
\end{proof}

In the last part of this section we determine the Weyl group $\mathsf W$, which will be used to determine the invariant metrics on $M^{13}_k$.

\begin{prop}\label{prop:Weylgroupelements}
The Weyl group of the cohomogeneity one action of $\sfG = \SO(2) \times \sfGt$ on $M^{13}_k$ is $\mathsf W \simeq \Zeit_2 \ltimes \Zeit_4$, which is generated by $w_- \in \Km$ and $w_+ \in \Kp$:
\begin{eqnarray*}
w_- & = & (i, A) \quad \text{with }  A =\diag\set{\begin{pmatrix}0 & \eps \\ - \eps & 0\end{pmatrix}, 1, \begin{pmatrix}0 & 0 & 0 & -\eps \\ 0& 1 & 0 & 0\\ 0 & 0 & 1 & 0 \\ \eps & 0 & 0 & 0\end{pmatrix}} \\
w_+ & = & \left(1, \diag\set{1, -1, -1, 1, 1, -1, -1}\right),
\end{eqnarray*}
where $\eps = 1$ for $k = 1, 5, \ldots$, and $\eps = -1$ for $k = 3, 7, \ldots$.
\end{prop}
\begin{proof}
First, it is easy to check that $w_+ \in \Kp$ and neither of $w_\pm$ is in $\sfH$. We show that $w_- \in \Km$. It is sufficient to prove that $A\in \sfGt$. Since $e^{i\theta} = i$, we may assume that $\theta = \frac{\pi}{2}$. It follows that $\eps = \sin k \theta$. Let $j$ be the complex structure induced by the left multiplication of $e_3$. So we have 
\[
A|_{\Span_\Real\set{e_1, e_2}} = j^k, \quad A|_{\Span_\Real\set{e_4, e_7}} = -j^k\quad  \text{and} \quad A|_{\Span_\Real\set{e_6, e_5}} = 1,
\]
i.e., $A$ embeds in $\U(2) \subset \SU(3)_3$ with the image $\diag\set{j^k, -j^k, 1}$ and so $A\in \sfGt$.

We check that each $w_{\pm}$ is of order $2$: 
\[
w_-^2 = \left(-1, \diag\set{-1,-1, 1, -1, 1,1, -1}\right) \in \sfH
\]
and 
\[
w_+^2 = \left(1, I_7 \right) \in \sfH.
\]
This shows that $w_\pm$ are generators of the Weyl group. Next we determine the order of $w_- w_+$. Write $w_- w_+ = (i, B)$, and we have
\[
B = \diag\set{
\begin{pmatrix}
0 & -\eps \\
-\eps & 0 
\end{pmatrix}, -1, 
\begin{pmatrix}
0 & 0 & 0 & \eps \\
0 & 1 & 0 & 0 \\
0 & 0 & -1 & 0 \\
\eps & 0 & 0 & 0 
\end{pmatrix}
}.
\]
It follows that $B^2 =I_7$, the identity matrix. So we have $(w_+w_-)^2 = (-1,I_7)\not\in \sfH$, but $(w_+ w_-)^4 = (1,I_7)\in \sfH$, i.e., $\mathsf W =\bk{w_-, w_+} \simeq \Zeit_2 \ltimes \Zeit_4$ which finishes the proof.
\end{proof}

\medskip{}
\section{The $\sfG$-invariant metrics on $M^{13}_k$}

In this section we determine all $\sfG$ invariant metric on $M^{13}_k$ with $\sfG = \SO(2) \times \sfGt$. See Proposition \ref{prop:invariantmetric} for the invariant metrics on the regular part, and Lemma \ref{lem:smoothmetric} for the conditions to ensure the smoothness of the metrics at the singular orbits. 

Throughout this section, we assume that $k$ is an odd integer. We refer to \cite[Section 1]{GroveZillerposRicci} for the description of invariant metrics on a general cohomogeneity one manifold. 

\smallskip

Recall that $c(t)$ is a normal minimal geodesic between two singular orbits $\Bm$ and $\Bp$; with $c(0) = p_-\in \Bm$, and $c(L) = p_+\in \Bp$. On the regular part of $M^{13}_k$, the metric is determined by 
\[
g_{c(t)} = dt^2 + g_t
\]
where $g_t$ is a family of homogeneous metrics on $\sfG /\sfH$. By means of Killing vector fields, we identify the tangent space of $\sfG/\sfH$ at $c(t)$, $t\in (0, L)$ with an $\Ad_{\sfH}$-invariant complement $\frakp$ of the isotropy subalgebra $\frakh$ of $\sfH$ in $\frakg$, and the metric $g_t$ is identified with an $\Ad_{\sfH}$-invariant inner product on $\frakp$.

In the following, we introduce a few subspaces in $\frakp$ such that the invariant metric has a block-diagonal form. The Lie algebra $\frakg_2$ of $\sfGt$ has the following embedding in $\so(7)$:
\begin{equation}
X = 
\begin{pmatrix}
0 & x_1 - y_1 & x_2 + y_2 & - x_5  + y_5 & - x_6 - y_6 & x_3 + y_3 & x_4 - y_4 \\
-x_1 + y_1 & 0 & b & y_4 & y_3 & y_6 & y_5 \\
-x_2 -y_2 & -b & 0 & x_3 & x_4 & x_5 & x_6 \\
x_5 -y_5  & -y_4  & -x_3 & 0  & a & y_2 & y_1 \\
x_6 +y_6 & -y_3 & -x_4 & -a & 0 & x_1 & x_2 \\
-x_3 - y_3 & -y_6 & -x_5 & -y_2 & -x_1 &  0 & a+ b \\
-x_4+y_4 & -y_5 & -x_6 & -y_1 & -x_2 &  -a-b & 0 
\end{pmatrix}
\end{equation}
for $a, b, x_1, \ldots, x_6, y_1, \ldots, y_6 \in \Real$. We choose the following bi-invariant inner product on $\frakg_2$:
\begin{eqnarray*}
Q_0(X, X) & = & -\frac{1}{4}\Tr X^2 \\
& = & a^2 + ab + b^2 + \sum_{i=1}^6 \left(x_i^2 + y_i^2\right) - x_1 y_1 + x_2 y_2 + x_3 y_3 -x_4 y_4 - x_5 y_5 +x_6 y_6.
\end{eqnarray*}
The Lie algebra $\frakh$ of $\sfH = \Zeit_2 \cdot \SU(2)$ has the following form
\begin{equation}\label{eqn:Liealgh}
\frakh = \set{
\begin{pmatrix}
O_{3\times 3} & O_{3\times 4} \\
O_{4\times 3} & A _{4\times 4}
\end{pmatrix}
\quad \text{with}\quad A = 
\begin{pmatrix}
 0 & a & -x_2 & x_1 \\
-a & 0 & x_1 & x_2 \\
x_2 & - x_1 & 0 & a \\
-x_1 & -x_2 & -a & 0 
\end{pmatrix}
}
\end{equation}
where $O_{p\times q}$ is the zero matrix. The $Q_0$-orthogonal complement $\frakm$ of $\frakh$ is given by 
\[
\frakm = \set{X \in \frakg_2 : b + 2a = 0, x_1 + y_1 = 0, \text{ and } x_2 - y_2 = 0}.
\]
Note that, $\frakh \subset \so(4)$ is the standard embedding of $\su(2) \subset \so(4)$: 
\[
A_1 + i A_2  \mapsto 
\begin{pmatrix}
A_1 & - A_2 \\
A_2 & A_1 
\end{pmatrix}.
\]
Denote the following matrices in $\frakm$: 
\[
U_0 = \diag\left\{
\begin{pmatrix}
0 & -2 & 0 \\
2 & 0 & 0 \\
0 & 0 & 0 
\end{pmatrix}, 
\begin{pmatrix}
0 & 0 & 0 & 1 \\
0 & 0 & -1 & 0 \\
0 & 1 & 0 & 0 \\
-1 & 0 & 0 & 0 
\end{pmatrix}
\right\},
\]
\[
U_1 = \diag\left\{
\begin{pmatrix}
0 & 0 & 0 \\
0 & 0 & 2 \\
0 & -2 & 0 
\end{pmatrix}, 
\begin{pmatrix}
0 & -1 & 0 & 0 \\
1 & 0 & 0 & 0 \\
0 & 0 & 0 & 1 \\
0 & 0 & -1 & 0 
\end{pmatrix}
\right\},
\]
and 
\[
U_2 = \diag\left\{
\begin{pmatrix}
0 & 0 & 2 \\
0 & 0 & 0 \\
-2 & 0 & 0 
\end{pmatrix}, 
\begin{pmatrix}
0 & 0 & 1 & 0 \\
0 & 0 & 0 & 1 \\
-1 & 0 & 0 & 0 \\
0 & -1 & 0 & 0 
\end{pmatrix}
\right\}.
\]
Then we have
\[
Q_0(U_i, U_i) = 3 \quad \text{and}\quad Q_0(U_i, U_j) = 0 \quad \text{for}\quad 0\leq i \ne j \leq 2.
\]
Denote $\frakm$'s subspaces 
\[
\frakm_1 = \set{
\begin{pmatrix}
0 & 0 & 0 & 0 & 0 & 0 & 0 \\
0 & 0 & 0 & x_4 & -x_3 & -x_6 & x_5 \\
0 & 0 & 0 & x_3 & x_4 & x_5 & x_6 \\
0 & -x_4 & -x_3 &  0 & 0 & 0 & 0 \\
0 & x_3  & -x_4 & 0 & 0 & 0 & 0 \\
0 & x_6 & -x_5  & 0 & 0 & 0 & 0 \\
0 & -x_5 & -x_6 & 0 & 0 & 0 & 0
\end{pmatrix}
=x_3 E_1 + x_4 E_2 + x_5 E_3 + x_6 E_4 },
\]
and 
\[
\frakm_2 = \left\{
\frac{1}{\sqrt 3}
\begin{pmatrix}
0 & 0 & 0 & -2x_5 & -2x_6 & 2x_3 & 2x_4 \\
0 & 0 & 0 & -x_4 & x_3 & x_6 & -x_5 \\
0 & 0 & 0 & x_3 & x_4 & x_5 & x_6 \\
2x_5 & x_4 & -x_3 &  0 & 0 & 0 & 0 \\
2x_6 & -x_3  & -x_4 & 0 & 0 & 0 & 0 \\
-2x_3 & -x_6 & -x_5  & 0 & 0 & 0 & 0 \\
-2x_4 & x_5 & -x_6 & 0 & 0 & 0 & 0
\end{pmatrix}   = x_3 F_1 + x_4 F_2 + x_5 F_3 + x_6 F_4\right\}.
\]
Note that our matrices of $E_1, \ldots E_4$ and $F_1, \ldots, F_4$ are different from those in \cite{GVWZ}. We have $Q_0(E_p, F_q) = 0$ for $1\leq p, q \leq 4$, and 
\begin{eqnarray*}
Q_0(E_i, E_i) = 1 & & Q_0(E_i, E_j) = 0 \\
Q_0(F_i, F_i) = 1 & & Q_0(F_i, F_j) = 0
\end{eqnarray*}
for $1\leq i\ne j \leq 4$.

Next, we consider the Lie algebra $\frakg = \so(2) \oplus \frakg_2$ with the following bi-invariant inner product
\begin{equation}
Q(s E_{12} + X, s E_{12} + X) = \frac{3k^2}{4}s^2 + Q_0(X, X)
\end{equation}
where $s E_{12} \in \so(2)$, and $E_{12}$ is the skew-symmetric $2\times 2$-matrix with $(2,1)$-entry $1$. So we have
\begin{equation}
\frakp = \so(2) + \frakm.
\end{equation}
Let
\begin{eqnarray}
X_1 = \left(\frac{2}{k} E_{12} + U_0\right)/\sqrt{6}, & &  X_2 = \left(\frac{2}{k} E_{12} - U_0\right)/\sqrt{6} \label{eqn:X1X2} \\
Y_1 = U_1/\sqrt 3, & & Y_2 = U_2/\sqrt 3. \label{eqn:Y1Y2}
\end{eqnarray}
It follows that $\set{X_1, X_2, Y_1, Y_2, E_1, \ldots, E_4, F_1, \ldots, F_4}$ is a $Q$-orthonormal basis of $\frakp$, and 
\begin{eqnarray*}
\frakk^- = \frakh + \Span_\Real\set{X_1}, & & T_{c(0)} B_- \simeq \frakm_1 + \frakm_2 + \Span_\Real\set{X_2, Y_1, Y_2} \\
\frakk^+ = \frakh + \frakm_1 + \Span_\Real \set{Y_1}, & & T_{c(L)} B_+ \simeq \frakm_2 + \Span_\Real\set{X_1, X_2, Y_2}. 
\end{eqnarray*}

From the explicit forms of the generators of the Weyl group $\mathsf W$ in Proposition \ref{prop:Weylgroupelements}, we determine the action of $\mathsf W$ on each subspace in $\frakp$. 
\begin{lem}
The action of the Weyl group $\mathsf W$ is given by the following:
\begin{enumerate}
\item $\Ad_{w_-}$ acts on $\frakp$ via
\[
X_1 \mapsto X_1, \quad X_2 \mapsto X_2, \quad Y_1 \mapsto \eps Y_2, \quad Y_2 \mapsto - \eps Y_1
\]
and
\begin{eqnarray*}
E_1 \mapsto \frac{\eps}{2} E_4 + \frac{\sqrt 3 \eps}{2} F_4, & & F_1 \mapsto \frac{\sqrt 3 \eps}{2} E_4 - \frac \eps 2 F_4 \\
E_2 \mapsto \frac 1 2 E_2 + \frac{\sqrt {3}}{2} F_2,  & & F_2 \mapsto \frac{\sqrt 3}{2}E_2 - \frac 1 2 F_2 \\
E_3 \mapsto \frac 1 2 E_3 + \frac{\sqrt{3}}{2} F_3, & & F_3 \mapsto \frac{\sqrt 3}{2}E_3 - \frac 1 2 F_3 \\
E_4 \mapsto -\frac \eps 2 E_1 - \frac{\sqrt{3}\eps}{2} F_1, & & F_4 \mapsto -\frac{\sqrt 3\eps}{2}E_1 + \frac \eps 2 F_1.
\end{eqnarray*}
\item $\Ad_{w_+}$ acts on $\frakp$ via
\[
X_1\mapsto X_2, \quad X_2 \mapsto X_1, \quad Y_1 \mapsto Y_1, \quad Y_2 \mapsto - Y_2 
\]
and 
\begin{eqnarray*}
& E_1 \mapsto - E_1, \quad E_2 \mapsto - E_2, \quad E_3 \mapsto E_3, \quad E_4 \mapsto E_4; & \\
& F_1 \mapsto - F_1, \quad F_2 \mapsto - F_2, \quad F_3 \mapsto F_3, \quad F_4 \mapsto F_4. & 
\end{eqnarray*}
\end{enumerate}
\end{lem}

We determine the irreducible summands of the $\Ad_\sfH$ representation on $\frakp$ in the following

\begin{lem}\label{lem:AdHrep}
The adjoint representation of $\sfH$ on the space $\frakp$ is determined by the following:
\begin{enumerate}
\item  For the connected component $\sfH_0 = \SU(2) \subset \sfH$, the representation of $\Ad_{\sfH_0}$ on 
\[
\frakp = \Span_\Real \set{X_1, X_2, Y_1, Y_2} \oplus \frakm_1\oplus \frakm_2
\] 
is given by
\[
1\oplus 1\oplus 1\oplus 1\oplus[\mu_2]_\Real \oplus [\mu_2]_\Real 
\]
where $1$ is the trivial representation, and $[\mu_2]_\Real$ is the standard representation of $\SU(2)$ on $\Cpx^2 = \Real^4$.
\item The element 
\[
\tau = \left(-1,\diag\set{-1,-1,1, -1, 1,1,-1}\right) \in \sfH 
\]
acts trivially on $\Span_\Real \set{X_1, X_2, E_2,E_3, F_2, F_3}$, \\
and maps $v$ to $-v$ on $\Span_\Real \set{Y_1, Y_2, E_1, E_4, F_1, F_4}$.
\end{enumerate}
\end{lem}
\begin{proof}
First note that the adjoint representation of $\sfH$ is trivial on the line spanned by $E_{12} \in \so(2)$. Recall that from the embedding (\ref{eqn:Liealgh}) of the Lie algebras, the identification between $\SU(2)$ and $\sfH_0=\SU(2) \subset  \SU(3) \subset\SO(7)$ is given by 
\[
\begin{pmatrix}
\alpha & \beta \\
-\bar{\beta} & \bar \alpha 
\end{pmatrix}
\mapsto 
h = \diag\set{
I_3, 
\begin{pmatrix}
h_1 & - h_2 \\
h_2 & h_1
\end{pmatrix}}
\]
with 
\[
h_1 = 
\begin{pmatrix}
a_1 & b_1 \\
-b_1 & a_1
\end{pmatrix}\quad \text{and}\quad 
h_2 = 
\begin{pmatrix}
-b_2 & a_2 \\
a_2 & b_2
\end{pmatrix}
\]
where $\alpha = a_1 + i a_2$, $\beta = b_1 + i b_2$ and the complex structure is induced by the left multiplication of ${e_3}$. It is straightforward to check that $\Ad_h U_j=U_j$ for $j=0,1,2$ and the following relations
\begin{eqnarray*}
\Ad_h 
\begin{pmatrix}
E_1 & F_1\\
E_2 & F_2\\
E_3 & F_3\\
E_4 & F_4
\end{pmatrix} = h^T
\begin{pmatrix}
E_1 & F_1\\
E_2 & F_2\\
E_3 & F_3\\
E_4 & F_4
\end{pmatrix}.
\end{eqnarray*}
This shows the first part. The statement in the second part follows by a straightforward computation.
\end{proof}

Denote $X^*$, the Killing vector field generated by $X \in \frakp$ along $c(t)$. Using the fixed background inner product $Q$ on $\frakp$, the invariant metric $g_t$, $t \in (0, L)$ can be written as 
\[
g_t(X^*, Y^*) = Q(P(t) X, Y)
\]
for any $X, Y \in \frakp$, where $P(t)$ is a family of positive definite $\Ad_{\sfH}$-invariant endomorphisms of $\frakp$. From Lemma \ref{lem:AdHrep} and Schur's Lemma in representation theory, we have

\begin{prop}\label{prop:invariantmetric}
Restricted to the regular part $M_k^{13}-\pset{\Bp\cup\Bm}$, a $\sfG$-invariant metric $g = dt^2+g_t$ is determined by the following inner products on the tangent space of $T_{c(t)} \sfG/\sfH\cong\frakp$ ($0< t< L$):
\begin{eqnarray*}
g_t(X_1, X_1) = f_1^2(t), & g_t(X_2, X_2) = f_2^2(t), & g_t(X_1, X_2) = f_{12}(t) \\
g_t(Y_1, Y_1) = h_1^2(t), & g_t(Y_2, Y_2) = h_2^2(t), & g_t(Y_1, Y_2) = h_{12}(t)  
\end{eqnarray*}
\begin{eqnarray*}
& g_t(E_i, E_i) = a_{1}^2(t), \quad g_t(F_i, F_i) = a_{2}^2(t), \quad g_t(E_i, F_i) = a_{12}(t) & \\
& g_t(E_1, F_4) = g_t(E_3, F_2) = b_{12}(t), \quad g_t(E_2, F_3) = g_t(E_4, F_1) = - b_{12}(t), & 
\end{eqnarray*}
with $i=1,\ldots, 4$, and the other components vanish. Here the $10$ functions are smooth on $(0,L)$ and $g_t$ is positive definite for any $t\in (0,L)$.
\end{prop}

\begin{rem}\label{rem:M13deveninvariantmetrics}
If $k$ is an even integer, from Remark \ref{rem:M13deven}, the principal isotropy subgroup is $\sfH = \Zeit_2 \times \SU(2)$, and the adjoint representation of $\sfH$ on $\frakp$ is given by Case (1) in Lemma \ref{lem:AdHrep}. It follows that for an invariant metric on the regular part, we need $10$ smooth functions to describe the inner products on $\Span_\Real\set{X_1, X_2, Y_1, Y_2}$, other $6$ smooth functions for the inner products on $\frakm_1 \oplus \frakm_2$.
\end{rem}
\begin{rem}
If the group is $\SO(2) \times \SO(7)$, there are $6$ functions involved for an invariant metric on $M_k^{13}$, see \cite{BackHsiang} and \cite{GVWZ}. 
\end{rem}

There are further conditions required such that the metric $dt^2 + g_t$ can be extended smoothly to singular orbits at $t=0$ and $L$. These conditions are given in \cite{BackHsiang} and \cite{GVWZ} when the group is $\SO(2) \times \SO(7)$. For our case with $\sfG = \SO(2) \times \sfGt$, we have 
\begin{lem}\label{lem:smoothmetric}
Assume $k\geq 3$ odd. To ensure the metric $g = dt^2 + g_t$ can be smoothly extended to the singular orbits at $t = 0$ and $L$, the following conditions hold.
\begin{eqnarray*}
& f_1(0) = 0, \quad f_{12}(0) = 0, \quad h_1(0) = h_2(0) > 0, \quad h_{12}(0) = 0, & \\
& a_{12}(0) = \frac{\sqrt{3}}{2}\left(a_1^2(0) - a_2^2(0)\right), \quad b_{12}(0) =0, & \\
& f_1'(0) = \frac{4}{k \sqrt 6}, \quad f_{12}'(0) = 0, \quad f_2'(0) = 0, \quad h_1'(0) = h_2'(0) = h_{12}'(0) = 0, & \\
& a_1'(0) = a_2'(0) = a_{12}'(0) = b_{12}'(0) = 0; &  
\end{eqnarray*}
and
\begin{eqnarray*}
& h_2(L) = a_2(L) >  0, \quad h_2'(L) = a_2'(L) = 0, \quad h_1(L) = a_1(L) = 0.& 
\end{eqnarray*}
\end{lem}
\begin{proof}
We first consider the singular orbit at $t=0$. Note that $\sigma = (e^{ i2\pi/k}, \Id) \in \Km$ acts trivially on $B_- = \sfG/\Km$, and the slice representation on the $2$-disk bundle of $B_-$ is given by $R(2\theta)$ for $R(\theta)\in \SO(2)$. Here $R(\phi)$ for $\phi \in [0, 2\pi)$ is the counterclockwise rotation with the matrix form 
\[
R(\phi) = 
\begin{pmatrix}
\cos \phi & \sin \phi \\
-\sin \phi & \cos \phi
\end{pmatrix}.
\] 
It follows that the singular orbit $B_-$ is the fixed points set of $\sigma$ and hence totally geodesic, see also \cite[p. 162]{GVWZ}.  

Since $X_1$ collapses on $B_-$, we have $f_1(0) = 0$ and $f_{12}(0) = 0$. The isotropy representation of $\Km = \SO(2) \SU(2)$ on the tangent space of 
\[
T_{c(0)}B_- = \Span_\Real \set{X_2} + \Span_\Real\set{Y_1, Y_2} + \frakm_1+\frakm_2
\] 
is given by 
\[
1 + \rho_2 \otimes 1 + \rho_2 \otimes [\mu_2]_\Real
\] 
where $\rho_2$ is the standard action of $\SO(2)$ on $\Real^2$ via $R(k \theta)$. Note that the third component above is not irreducible as a real representation. That the second component is irreducible as a real representation, implies that 
\[
h_1(0) = h_2(0) > 0, \quad h_{12}(0) = 0.
\]
In the following we consider the representation on $\frakm_1 + \frakm_2$. An explicit matrix form of the $\SO(2)$ action on $\Im \Cay = \Span_\Real\set{e_1, \ldots, e_7}$ is given by
\[
A = \diag\set{
\begin{pmatrix}
\cos 2u & - \sin 2u & 0 \\
\sin 2u & \cos 2u & 0 \\
0 & 0 & 1 \\
\end{pmatrix},
\begin{pmatrix}
\cos u & 0 & 0 & \sin u \\
0 & \cos u & - \sin u & 0 \\
0 & \sin u & \cos u & 0 \\
-\sin u & 0 & 0 & \cos u
\end{pmatrix}
}
\]
with $u =  -k \theta/2$. The adjoint action $\Ad_A$ on $\frakm_1 + \frakm_2$ under the basis $\set{E_1, \ldots, E_4, F_1, \ldots, F_4}$ has the matrix form $M =\left( M_1 | M_2\right)$, with
\[
M_1 = 
\begin{pmatrix}
 \cos^3 u & 0 & 0 & \sin^3 u  \\
 0 & \cos^3 u & -\sin^3 u & 0 \\
 0 & \sin^3 u & \cos^3 u & 0  \\
 -\sin^3 u & 0 & 0 & \cos^3 u  \\
 \sqrt{3} \cos u \sin^2 u & 0 & 0 & \sqrt{3} \cos^2 u \sin u  \\
 0 & \sqrt{3} \cos u \sin^2 u & -\sqrt{3} \cos^2 u \sin u & 0 \\
 0 & \sqrt{3} \cos^2 u \sin u & \sqrt{3} \cos u \sin^2 u & 0 \\
 -\sqrt{3} \cos^2 u \sin u & 0 & 0 & \sqrt{3} \cos u \sin^2 u 
\end{pmatrix}
\]
and 
\[
M_2 = 
\begin{pmatrix}
\sqrt{3} \cos u \sin^2 u & 0 & 0 & \sqrt{3} \cos^2 u \sin u \\
0 & \sqrt{3} \cos u \sin^2 u & -\sqrt{3} \cos^2 u \sin u & 0 \\
0 & \sqrt{3} \cos^2 u \sin u & \sqrt{3} \cos u \sin^2 u & 0 \\
-\sqrt{3} \cos^2 u \sin u & 0 & 0 & \sqrt{3} \cos u \sin^2 u \\
(\cos u+3 \cos 3 u)/4 & 0 & 0 & (\sin u-3 \sin 3 u)/4 \\
0 & (\cos u+3 \cos 3 u)/4 & (-\sin u+3 \sin 3 u)/4 & 0 \\
0 & (\sin u-3 \sin 3 u)/4 & (\cos u+3 \cos 3 u)/4 & 0 \\
(-\sin u+3 \sin 3 u)/4 & 0 & 0 & (\cos u+3 \cos 3 u)/4
\end{pmatrix}.
\]
Using the same basis of $\frakm_1 + \frakm_2$, the endomorphism $P(t)$ has the following matrix form:
\[
P(t) = 
\begin{pmatrix}
a_1^2(t)I_4 & P_{12}(t) \\
P_{12}(t) & a_2^2(t) I_4
\end{pmatrix}
\quad \text{with} \quad 
P_{12}(t) = 
\begin{pmatrix}
a_{12}(t) & 0 & 0 & b_{12}(t) \\
0 & a_{12}(t) & - b_{12}(t) & 0 \\
0 & b_{12}(t) & a_{12}(t) & 0 \\
-b_{12}(t) & 0 & 0 &  a_{12}(t)
\end{pmatrix}
\]
where $I_4$ is the identity matrix. So the $\Km$ invariance of $P(0)$, i.e., $M P(0) = P(0) M$, implies that 
\[
b_{12}(0) = 0 \quad \text{and}\quad a_{12}(0) = \frac{\sqrt 3}{2}\left(a_1^2(0) - a_2^2(0)\right).
\]

Note that on the circle $R(\theta)(0\leq \theta \leq 2\pi)$, we have $R(\pi) \in \sfH$. So we have $\phi'(0) = 2$, with $\phi(t)$ the length of Killing vector field generated by $\frac{d}{d\theta}$. By our choice of $X_1$, we have $f_1(t) = \frac{2}{k \sqrt 6} \phi(t)$ so that $f_1'(0) = \frac{4}{k\sqrt 6}$. Since $\Ad_{w_-}$ fixes $X_1$ and $X_2$, we have $g_t(X_1^*, X_2^*)$ is invariant under the reflection of the $2$-disk slice generated by $\Ad_{w_-}$ that changes $t$ to $-t$. It follows that $f_{12}'(0) = 0$. Similarly we also have $f_2'(0) = 0$. The other derivatives vanish at $t=0$ follows from the fact that $B_-$ is totally geodesic and the second fundamental form is $-\frac 1 2 P_t^{-1}P_t'$.

Next we consider the singular orbit at $t=L$. The slice at $p_+$ is $V = \Real^6$, and the action by the connected component $\Kp_0=\SU(3)$ is given by $[\mu_3]_\Real$. Restricted to the subspace $W =\Span_\Real\set{U_0, U_2} \oplus \frakm_2 \subset T_{c(L)}B_+$, the adjoint representation by $\Kp_0$ is given by $[\mu_3]_\Real$. So we have $h_2(L) = a_2(L)$. The second fundamental form $\mathrm{II}$ at $c(L)$ restricted on $W\times W$ is a $\Kp_0$-equivariant map  
\[
\mathrm{II}: \mathrm{Sym}^2(W)\times V \To \Real.
\]
However the symmetric square of $[\mu_3]_\Real$ is given by $[2,0]_\Real  \oplus [1,1]\oplus 1$ in terms of highest weight notions, and it does not contain $[\mu_3]_\Real = [1,0]_\Real$. It follows that $\mathrm{II}$ restricted on $W\times W$ vanishes at $c(L)$ and so we have $a_2'(L) = h_2'(L) = 0$. The equations $a_1(L) = h_1(L) = 0$ follow from the fact that $Y_1$ and $\frakm_1$ collapse at $c(L)$. This finishes the proof
\end{proof}

\medskip{}
\section{Rigidities of non-negatively curved metrics}

In this section, we derive a few rigidity results when the invariant metric is assumed to be non-negatively curved, see Propositions \ref{prop:parallelJacobi} and \ref{prop:a1squarebounds}.

\smallskip

Recall the following rigidity result on Jacobi vector fields in \cite{VerdianiZiller}.
\begin{prop}[{\cite[Proposition 3.2]{VerdianiZiller}}]\label{prop:parallelJacobiVZ}
Let $M^{n+1}$ be a manifold with non-negative sectional curvature,  and $V$ a self adjoint family of Jacobi fields along the geodesic $c: [t_0, t_1] \To M$. Assume there exists an $X \in V$ such that the following conditions hold. 
\begin{enumerate}[(a)]
\item $\norm{X}_t \ne 0$, $\norm{X}'_{t}=0$ for $t = t_0$ and $t=t_1$.
\item If $Y \in V$ and $\bk{X(t_1), Y(t_1)} = 0$, then $\bk{X(t_0), Y(t_0)} = 0$.
\item If $Y \in V$ and $Y(t) = 0$ for some $t\in (t_0, t_1)$, then $\bk{X(t_0), Y(t_0)} = 0$.
\item If $Y(t_0) = 0$, then $\bk{X'(t_0), Y'(t_0)} = 0$.
\end{enumerate}
Then $X$ is a parallel Jacobi vector field along $c$.
\end{prop}

We consider the case where $V$ is given by a family of Killing vector fields. Recall that for any $X\in \frakg$, $X^*$ is the Killing vector field generated by $X$ along the geodesic $c(t)$, and denote $X(t) = X^*(t)$. Since the parallel transport along $c(t)$ is $\Ad_\sfH$-invariant, we may choose $V = \set{X^* : X \in \mathfrak n}$ for the subspace $\mathfrak n \subset \frakp$ such that it is the sum of all equivalent irreducible representations in $\frakp$. 

We show that such $V$ is a self adjoint family of Jacobi fields along the geodesic $c(t)$. Let $T = \frac{\partial}{\partial t}$  be the unit tangent vector along $c(t)$. For any $X^*, Y^* \in V$ we have
\begin{eqnarray*}
g(\nabla_T X^*, Y^*) & = & - g(\nabla_{Y^*}X^*, T) = - g(\nabla_{X^*}Y^*, T) \\
& = & g(\nabla_{X^*}T, Y^*),
\end{eqnarray*}
and 
\begin{eqnarray*}
g(X'(t), Y(t)) & = & g(\nabla_T X(t), Y(t)) = g(\nabla_{X(t)} T, Y(t)) = g(\nabla_{Y(t)}T, X(t)) \\
& = & g(Y'(t), X(t)).
\end{eqnarray*}
So $V$ is self-adjoint. We also have 
\begin{eqnarray*}
g(X'(t), Y(t)) & = & \frac{1}{2}D_T g(X(t), Y(t)) = \frac{1}{2} Q(P'(t) X, Y)
\end{eqnarray*}
and thus
\begin{equation}\label{eqn:XprimeP}
X'(t) = \frac{1}{2} P(t)^{-1} P'(t) X.
\end{equation}

\begin{prop}\label{prop:parallelJacobi}
Suppose that $(M^{13}_k, g)$ has non-negative curvature with $g$ an invariant metric and $k\geq 3$ odd. The Killing vector fields $X^*$ generated by the following vectors $X \in \frakp$ are parallel Jacobi fields along $c(t)$($t\in [0, L]$):
\[
X = Y_2
\]
and
\[
X = \beta E_i + F_i \, (i=1,2,3,4)\quad \text{with}\quad \beta = - \frac{a_{12}(0)}{a_1^2(0)}.
\]
Moreover for all $t\in [0, L]$, we have $h_{12}(t) = b_{12}(t) = 0$ and 
\[
h_{2}(t) = h_2(L) > 0, \quad a_{12}(t) = - \beta a_1^2(t), \quad a_2^2(t) = \beta^2 a_1^2(t) + h_2^2(L).
\]
\end{prop}
\begin{proof}
We first consider the case $X = Y_2$. By $\Ad_\sfH$-invariance take 
\[
V = \set{Y^* : Y \in \Span_\Real\set{Y_1, Y_2}}.
\] 
In Proposition \ref{prop:parallelJacobiVZ}, condition (a) holds as $h_2(t) \ne 0$ and $h_2'(t) =0$ at $t=0$ and $L$. For condition (b), if $g(Y_2(L), Y(L)) = 0$, then $Y = \lambda Y_1$ for some constant $\lambda$. So (b) holds as 
\[
g(Y_2(0), \lambda Y_1(0)) = \lambda h_{12}(0) = 0.
\]
Condition (c) and (d) hold as such $Y$ is zero in $V$. It follows that $Y_2^*$ is a parallel Jacobi field for $t\in [0, L]$, $h_2(t)$ is a constant function and $h_{12}(t) = 0$ for $t\in [0,L]$.

Next for the case $X = F_i + \beta E_i$, we take $V = \set{Y^* : Y \in \frakm_1 +\frakm_2}$. We may assume that $i=1$. We have
\begin{eqnarray*}
\norm{X(t)}^2 & = & a_2^2(t) + \beta^2 a_1^2(t) + 2\beta a_{12}(t) \\
\norm{X(t)}\norm{X(t)}' & = & a_2'(t) + \beta^2 a_1'(t) + \beta a_{12}'(t).
\end{eqnarray*}
It follows that 
\begin{eqnarray*}
\norm{X(0)}^2 & = & a_2^2(0) + \beta^2 a_1^2(0) + 2\beta a_{12}(0) \\
& = & a_2^2(0) + \frac{a^2_{12}(0)}{a_1^2(0)} - 2\frac{a_{12}^2(0)}{a_1^2(0)} \\
& = & a_2^2(0) - \frac{a_{12}^2(0)}{a_1^2(0)} \\
& = & a_2^2(0) - \frac{3}{4}a_1^2(0)\left(1-\frac{a_2^2(0)}{a_1^2(0)}\right)^2.
\end{eqnarray*}
If $\norm{X(0)} = 0$, then we have 
\[
\frac{a_2^2(0)}{a_1^2(0)} = \frac{3}{4} \left(1-\frac{a_2^2(0)}{a_1^2(0)}\right)^2.
\]
It follows that either $a_1^2(0) = 3 a_2^2(0)$ or $a_2^2(0) = 3 a_1^2(0)$. Say $a_1^2(0) = 3 a_2^2(0)$, then Lemma \ref{lem:smoothmetric} implies that $a_{12}(0) = \sqrt{3} a_2^2(0)$ and then the Killing vector fields $E_1(0)$ and $F_1(0)$ are parallel which shows a contradiction. Similarly the second case cannot happen either and so we have $\norm{X(0)} \ne 0$. From Lemma \ref{lem:smoothmetric} again we have $\norm{X(0)}' = 0$. At $t=L$ since $E_1(L) = 0$ we have $\norm{X(L)} = a_2(L) > 0$, and $\norm{X(L)}' = a_2'(L) = 0$ from Lemma \ref{lem:smoothmetric}. So Condition (a) in Proposition \ref{prop:parallelJacobiVZ} holds for $X$. 

For Condition (b) in Proposition \ref{prop:parallelJacobiVZ}, we may assume that $Y = y_1 E_1 + y_2 F_1$. It follows that 
\[
\bk{X(L), Y(L)} = y_2 a_2^2(L)
\]
and $\bk{X(L), Y(L)}= 0$ implies that $y_2 = 0$. By normalization we assume that $Y = E_1$, and then  
\[
\bk{X(0), Y(0)} = \bk{F_1(0) + \beta E_1(0), E_1(0)} = a_{12}(0) + \beta a_1^2(0) = 0
\]
by our choice of $\beta$. So Condition (b) holds for $X$. Condition (c) and (d) also hold as such $Y$ is zero in $V$. It follows that the Killing vector field $X^*$ is a parallel Jacobi field for $t \in [0, L]$. Note that equation (\ref{eqn:XprimeP}) yields
\[
2X'(t) = P(t)^{-1} P'(t) X 
\]
and the block in $P(t)$ corresponding to $\set{E_1, F_1, E_4, F_4}$ is given by 
\[
P_1(t) = 
\begin{pmatrix}
a_1^2(t) & a_{12}(t) & 0 & b_{12}(t) \\
a_{12}(t) & a_2^2(t) &  - b_{12}(t) & 0 \\
0 & - b_{12}(t) & a_1^2(t) & a_{12}(t) \\
b_{12}(t) & 0 & a_{12}(t) & a_2^2(t) 
\end{pmatrix}.
\]
It follows that $P_1(t)^{-1} P_1'(t) X =0$ and then $P'_1(t) X = 0$, i..e, we have $b_{12}'(t) = 0$ and 
\begin{eqnarray*}
\frac{d}{dt}\left(\beta a_1^2(t) + a_{12}(t)\right) & = & 0 \\
\frac{d}{dt}\left(\beta a_{12}(t) + a_2^2(t)\right) & = & 0
\end{eqnarray*}
for any $t\in (0, L)$. So we have $b_{12}(t) = b_{12}(0) = 0$ and 
\begin{eqnarray*}
a_{12}(t) + \beta a_1^2(t) & = & a_{12}(0) + \beta a_1^2(0) = 0 \\
a_2^2(t) + \beta a_{12}(t) & = & a_2^2(L) - \beta a_1^2(L) = a_2^2(L).
\end{eqnarray*}
Note that $a_2(L) = h_2(L)$ and it finishes the proof. 
\end{proof}

In the following we assume that $h_2(L) = 1$ by rescaling the metric $g$ if necessary. From Proposition \ref{prop:parallelJacobi} and Lemma \ref{lem:smoothmetric} we have
\[
\beta = - \frac{a_{12}(0)}{a_1^2(0)},  \quad a_2^2(0) = \beta^2 a_1^2(0) + 1
\]
and
\[
a_{12}(0) = \frac{\sqrt 3}{2} \left(a_1^2(0) - a_2^2(0)\right).
\]
Solving $a_1^2(0)$ yields
\begin{equation}\label{eqn:a1sqaurebeta}
a_1^2(0) = \frac{\sqrt{3}}{\sqrt{3}(1-\beta^2) + 2\beta}.
\end{equation}
In particular we have $\beta \in \left(-\frac{1}{\sqrt{3}}, \sqrt 3\right)$.

\begin{prop}\label{prop:a1squarebounds}
Suppose that $(M^{13}_k, g)$ has non-negative curvature with $g$ an invariant metric and $k\geq 3$ odd. Assume that $h_2(L) = 1$. Then we have
\begin{equation}
\frac{3}{4}\leq a_1^2(0) \leq \frac{7}{12} + \frac{\sqrt{13}}{6} \approx 1.184.
\end{equation}
\end{prop}
\begin{proof}
The lower bound of $a_1^2(0)$ follows from the minimum value of the function $a_1^2(0)$ in equation (\ref{eqn:a1sqaurebeta}). To obtain the upper bound, we consider the sectional curvature of the $2$-plane spanned by $Y_1$ and $E_1 + r F_1$ on the singular orbit $B_-$. Note that $B_-$ is totally geodesic and a computation (see the details in Appendix A.1) yields
\begin{eqnarray*}
R(Y_1, E_1, E_1, Y_1) & = & \frac{6\sqrt{3} \beta^5 + 9 \beta^4 - 32\sqrt{3} \beta^3 + 10 \beta^2 + 18 \sqrt{3}\beta + 9}{4\left(\sqrt 3 \beta^2 - 2\beta - \sqrt 3\right)^2} \\
R(Y_1, F_1, F_1, Y_1) & = & \frac{27 \beta^4 + 12 \sqrt{3}\beta^3 + 22\beta^2 + 4\sqrt{3}\beta + 3}{12\left(\sqrt 3 \beta^2 - 2\beta - \sqrt 3\right)^2} 
\end{eqnarray*}
and
\begin{eqnarray*}
R(Y_1, E_1, F_1, Y_1) = - \frac{\beta\left(9 \beta^4 + 12\sqrt{3}\beta^3 - 54 \beta^2 + 20\sqrt{3}\beta + 57\right)}{12\left(\sqrt 3 \beta^2 - 2\beta - \sqrt 3\right)^2}.
\end{eqnarray*}
A necessary condition that $R(Y_1, E_1 + r F_1, E_1 + r F_1, Y_1)\geq 0$ for all $r$ is that 
\[
p(\beta) = R(Y_1, E_1, E_1, Y_1) R(Y_1, F_1, F_1, Y_1) - \left(R(Y_1, E_1, F_1, Y_1)\right)^2 \geq 0.
\]
From the formulas of the Riemann tensors we have
\[
p(\beta) = \frac{\left(\sqrt{3} \beta^2 + 2\beta - \sqrt{3}\right)\left(- 9 \beta^6 + 30 \sqrt{3}\beta^5 + 183 \beta^4 - 4\sqrt{3}\beta^3 - 183 \beta^2 + 30 \sqrt{3}\beta + 9\right)}{48\left(\sqrt{3} \beta^2 - 2\beta - \sqrt{3}\right)^3}.
\]
Note that $p(0) > 0$. On the interval $(-1/\sqrt{3}, \sqrt{3})$, the numerator of $p(\beta)$ has a simple root $\beta_1<0$ and a triple root $\beta_2>0$ given by
\[
\beta_1 = \frac{7}{3}\sqrt{3} - \frac{2}{3}\sqrt{39}  \quad \text{and}\quad \beta_2 = \frac{1}{\sqrt{3}}.
\]
So we have $\beta \in [\beta_1, \beta_2]$. Over this interval the function $a_1^2(0)$ is monotone decreasing with 
\[
a_1^2(0)\Big{|}_{\beta = \beta1} = \frac{7}{12}+\frac{\sqrt{13}}{6}\approx 1.184 \quad \text{and}\quad a_1^2(0)\Big{|}_{\beta = \beta2} =\frac{3}{4}.
\]
This finishes the proof.
\end{proof}

\medskip
\section{Proof of Theorem \ref{thm:introobstruction}}

We prove Theorem \ref{thm:introobstruction} in this section. Note that there is a shorter proof that works for $k\geq 5$, see Remark \ref{rem:shortproof}. 

\smallskip

Throughout this section we assume that $k\geq 3$ is an odd integer, and that $M^{13}_k$ admits an invariant metric $g$ with non-negative curvature. We assume that $h_2(L) = 1$ by rescaling the metric $g$ if necessary. It follows from Lemma \ref{lem:smoothmetric}, Propositions \ref{prop:parallelJacobi} and \ref{prop:a1squarebounds}, we have 
\begin{eqnarray*}
& b_{12}(t) = h_{12}(t) = 0, \quad h_2(t) = 1, & \\
& a_{12}(t) = - \beta a_1^2(t), \quad a_2^2(t) = \beta^2 a_1^2(t) + 1, & 
\end{eqnarray*}
for some constant $\beta$, and
\begin{eqnarray*}
& f_1(0) = 0, \quad f_{12}(0) = 0, \quad h_1(0) = 1, \quad  a_1^2(0) = \frac{\sqrt 3}{\sqrt 3(1-\beta^2)+2\beta}; & \\
& f_1'(0) = \frac{4}{k \sqrt 6}, \quad f_{12}'(0) = 0, \quad f_2'(0) = 0, \quad h_1'(0) = 0, \quad a_1'(0) = 0; &  
\end{eqnarray*}
\begin{eqnarray*}
& h_1(L) = a_1(L) = 0.& 
\end{eqnarray*}
The endomorphism has the following block-diagonal form
\begin{eqnarray*}
P
\begin{pmatrix}
X_1 \\
X_2
\end{pmatrix} & = & 
\begin{pmatrix}
f_1^2 & f_{12} \\
f_{12} & f_2^2
\end{pmatrix}
\begin{pmatrix}
X_1 \\
X_2
\end{pmatrix} \\
P
\begin{pmatrix}
Y_1 \\
Y_2
\end{pmatrix} & = &  
\begin{pmatrix}
h_1^2 & 0 \\
0 & 1
\end{pmatrix}
\begin{pmatrix}
Y_1 \\
Y_2
\end{pmatrix}
\end{eqnarray*}
and
\[
P
\begin{pmatrix}
E_i \\
F_i
\end{pmatrix} = 
\begin{pmatrix}
a_1^2 & -\beta a_{1}^2 \\
-\beta a_{1}^2 & \beta^2 a_1^2 + 1
\end{pmatrix}
\begin{pmatrix}
E_i \\
F_i
\end{pmatrix}
\quad \text{for } i =1,2,3,4.
\]

\begin{lem}\label{lem:a1h1pp}
We have $a_1''(t) \leq 0$ and $h_1''(t) \leq 0$ for $t\in [0, L]$.
\end{lem}
\begin{proof}
We know that $V = \Span_\Real\set{E_1, F_1}$ is an invariant space of $P(t)$ with the following matrix form
\[
P
\begin{pmatrix}
E_1 \\
F_1
\end{pmatrix}
=
\begin{pmatrix}
a_1^2 & - \beta a_1^2 \\
- \beta a_1^2 & \beta^2 a_1^2 + 1
\end{pmatrix}
\begin{pmatrix}
E_1 \\
F_1
\end{pmatrix}
\]
and the inverse is given by 
\[
P^{-1}\Big{|}_{V} =
\begin{pmatrix}
\beta^2 + \frac{1}{a_1^2} & \beta \\
\beta & 1
\end{pmatrix}.
\]
So the sectional curvature $K(E_1, T)$ of the plane spanned by $E_1$ and $T = \frac{\partial}{\partial t}$ has the same sign as
\[
R(E_1, T, T, E_1) = -a_1(t) a_1''(t). 
\]
The non-negativity of $K(E_1, T)$ implies that $a_1''(t) \leq 0$. The inequality of $h_1''(t)$ follows similarly from $K(Y_1, T)\geq 0$.
\end{proof}

Let
\begin{equation}\label{eqn:xia1}
\xi(t)=a_1^2(0) -a_1^2(t) 
\end{equation}
and from Lemma \ref{lem:a1h1pp}, we have 
\[
0\leq \xi(t) \leq a_1^2(0) \quad \text{for}\quad t\in [0, L]
\] 
and $\xi(0)=\xi'(0)=0$. 

\begin{lem}\label{lem:lupperbound}
The sectional curvature of the plane spanned by $X$ and $Y$ with 
\[
X= E_1 - \sqrt{3}F_1\quad \text{and}\quad Y =\sqrt{3}E_4 + F_4
\]
is given by
\[
K(X, Y) =\frac{R(X,Y,Y,X)}{\abs{X\wedge Y}^2}
\]
with
\begin{equation}\label{eqn:KXYxi}
\frac{4a_1^4(0)}{3} R(X,Y,Y,X) = \frac{8}{3} \frac{f_1^2 + f_2^2 + 2 f_{12}}{f_1^2 f_2^2 - f_{12}^2} \left(\xi(t)\right)^2 - \left(\xi'(t)\right)^2.
\end{equation}
Moreover $K(X, Y)\geq 0$ implies that 
\begin{equation}\label{eqn:xipxiupperbound}
\frac{f_1 \xi'}{\xi} \leq (1+\eta(t)) \frac{2\sqrt{6}}{3}\quad \text{for} \quad t\in (0, L),
\end{equation}
where $\eta(t)$ is a positive function with $\lim_{t\To 0}\eta(t) = 0$.
\end{lem}
\begin{proof}
The formula of $R(X, Y, Y, X)$ in equation (\ref{eqn:KXYxi}) is derived in Appendix A.2. To get inequality (\ref{eqn:xipxiupperbound}), one can apply the initial conditions $f_1(0) =f_{12}(0) = 0$ and $f_2(0) >0$.
\end{proof}
\begin{rem}
The choice of such vectors $X$ and $Y$ is motivated by Lemma 1.1(b) in \cite{WilkingZiller}. Here $X$ and $Y$ are eigenvectors of $P(0)$. The sectional curvature of the 2-plane is zero at $t=0$, and the contribution to the sectional curvature from the second fundamental form for $t>0$ involves the function $f_1$.  
\end{rem}

In the proof of Theorem \ref{thm:introobstruction}, the following algebraic fact of certain quartic functions is also needed. Denote
\begin{equation}\label{eqn:alphagamma}
\alpha = a_1^2(0) \quad \text{and} \quad \gamma = \sqrt{\alpha (4\alpha - 3)}
\end{equation}
and we introduce the following two quartic functions
\begin{eqnarray*}
\Psi_1(x) & = & \frac{5\alpha + 2\gamma}{48 \alpha^2} x^4 + \frac{2\alpha - \gamma}{24\sqrt{3}\alpha^2}x^3 - \frac{\alpha + \gamma}{8\alpha^2} x^2 + \frac{2\alpha +\gamma}{8\sqrt{3}\alpha^2} x - \frac{1}{16\alpha} \\
\Psi_2(x) & = & \frac{3\alpha^2 - \alpha -2\gamma}{48\alpha^2}x^4 + \frac{2\alpha^2 - 3\alpha +\gamma}{8\sqrt{3}\alpha^2}x^3 + \frac{9-2\alpha}{48\alpha}x^2 - \frac{1}{4\sqrt{3}} x+ \frac{1}{16}.
\end{eqnarray*}

\begin{lem}\label{lem:quarticfunctions}
Assume $\alpha\geq \frac 3 4$. Then we have
\[
3\Psi_1(x) +4\Psi_2(x) \geq 0
\] 
for any $x\in \Real$. Moreover the minimum can be achieved by a unique $x = x_\alpha$ such that $\Psi_2(x_\alpha) > 0$.
\end{lem}

\begin{proof}
Denote $\Psi(x) = 3\Psi_1(x) + 4\Psi_2(x)$. First we show that $\Psi(x) = 0$ has a double real root. One may see the fact from the vanishing of the discriminant. In the following we solve this double root explicitly. A calculation yields
\begin{eqnarray*}
\Psi(x) & = & \frac{11\alpha + 12 \alpha^2 -2 \gamma}{48 \alpha^2} x^4 + \frac{-10\alpha + 8 \alpha^2 + 3\gamma}{8\sqrt{3}\alpha^2}x^3 + \frac{9\alpha - 4\alpha^2 - 9 \gamma}{24\alpha^2} x^2 \\
& & + \frac{6\alpha - 8\alpha^2 + 3\gamma}{8\sqrt{3}\alpha^2} x + \frac{4\alpha - 3}{16\alpha} \\
\Psi'(x) & = & \frac{11\alpha + 12 \alpha^2 - 2\gamma}{12\alpha^2}x^3 + \frac{\sqrt{3}\left(-10\alpha + 8\alpha^2 + 3\gamma\right)}{8\alpha^2} x^2 + \frac{9\alpha - 4\alpha^2 - 9\gamma}{12\alpha^2} x \\
& & + \frac{6\alpha - 8\alpha^2 +3\gamma}{8\sqrt{3}\alpha^2} \\
\Psi''(x) & = & \frac{11\alpha+12\alpha^2 - 2\gamma}{4\alpha^2} x^2 + \frac{\sqrt{3}\left(-10\alpha + 8\alpha^2 + 3\gamma\right)}{4\alpha^2} x + \frac{9\alpha - 4\alpha^2 - 9\gamma}{12\alpha^2}.
\end{eqnarray*}
One can check that the following $x_\alpha$ is a common real root of $\Psi(x) = \Psi'(x) = 0$:
\begin{equation}\label{eqn:xalpha}
x_\alpha = \frac{\sqrt{3}\left(3-4\alpha -4 \gamma\right)}{3+12\alpha}
\end{equation}
and $\Psi''(x) = \frac{8}{3} - \frac{3}{2\alpha} >0$. It follows that $x_\alpha$ is a local minimum of $\Psi(x)$.  

Write
\[
\Psi(x) = \frac{11\alpha + 12\alpha^2 - 2\gamma}{48\alpha^2} (x-x_\alpha)^2 p(x)
\] 
and then we have
\[
p(x) = x^2 - \frac{2\sqrt{3}\left(2-\alpha +\gamma\right)}{4+3\alpha} x + \frac{3\alpha}{5\alpha + 2\gamma}.
\]
The discriminant $\Delta$ of $p(x)$ is given by
\[
\Delta = \frac{36}{12-41\alpha - 20 \gamma} < 0
\]
that implies that $\Psi(x) =  0$ has no other real roots. 

To finish the proof we only need to check that $\Psi_2(x_\alpha) > 0$. An explicit computation shows that 
\[
\Psi_2(x_\alpha) = \frac{(16\alpha - 9)\left(9 - 312 \alpha + 656 \alpha^2 - 48\gamma + 320 \alpha \gamma\right)}{36\alpha(1+4\alpha)^4} >0
\]
as $\alpha \geq \frac 3 4$.
\end{proof}

We will use the sectional curvature of the plane spanned by $A_r= X_1 + r X_2$ and $B_q= E_1 + q F_1$. Let
\begin{eqnarray*}
R_1 = R(X_1, E_1, E_1, X_1) & & R_2 = R(X_1, E_1, F_1, X_1) \\
R_3 = R(X_1, F_1, F_1, X_1) & & R_4 = R(X_2, E_1, E_1, X_2) \\
R_5 = R(X_2, E_1, F_1, X_2) & & R_6 = R(X_2, F_1, F_1, X_2) \\
R_7 = R(X_1, E_1, E_1, X_2) & & R_8 = R(X_1, F_1, E_1, X_2) \\
R_9 = R(X_1, E_1, F_1, X_2) & & R_{10} = R(X_1, F_1, F_1, X_2). 
\end{eqnarray*}
The formulas of $R_i$'s are listed in Appendix A.3. In the following, we group the terms in $R_i$'s into three different parts: one with the factor $\xi$, with the factor $\xi'$, and without the factor $\xi$ or $\xi'$.

\begin{lem}\label{lem:Ri}
The $R_i$'s have the following forms:
\begin{eqnarray*}
R_1 & = & -\frac{\xi}{2\alpha}\left(1+\eta_1\right) + \frac{1}{2} f_1 f_1' \xi' + \frac{1}{8}\left(f_1^2 - f_{12}\right)^2 \\
R_2 & = & \frac{\xi}{2\sqrt{3} \alpha}\left(1+\eta_2 \right) + \frac{1}{2\sqrt{3}}\left(\frac{\gamma}{\alpha} - 1\right)f_1 f_1' \xi' - \frac{1}{8\sqrt{3}}\left(1 + \frac{\gamma}{\alpha}\right)\left(f_1^2 - f_{12}\right)^2 \\
(\alpha - \xi)R_3 & = & \frac{\xi}{2}\left(1+\eta_3 \right) + \frac{5\alpha -2\gamma -3}{6}f_1 f_1' \xi' + \frac{5\alpha + 2\gamma}{24}\left(f_1^2 - f_{12}\right)^2 \\
R_4 & = & \frac{-2+f_2^2(0)}{4\alpha} \xi \left(1+\eta_4\right) + \frac{1}{2} f_2 f_2' \xi' + \frac{1}{8}\left(f_2^2 - f_{12}\right)^2 \\
R_5 & = & \frac{2-f_2^2(0)}{4\sqrt{3} \alpha} \xi \left(1+\eta_5 \right) + \frac{1}{2\sqrt{3}}\left(\frac{\gamma}{\alpha} - 1\right)f_2 f_2' \xi' - \frac{1}{8\sqrt{3}}\left(1 + \frac{\gamma}{\alpha}\right)\left(f_2^2 - f_{12}\right)^2 \\
(\alpha - \xi)R_6 & = & \pset{\frac{1}{2} - \frac{f_2^2(0)}{4} - \frac{5\alpha +2\gamma -3}{24\alpha}f_2^4(0)} \xi \left(1+\eta_6 \right) + \frac{5\alpha -2\gamma -3}{6}f_2 f_2' \xi' \\
& &  + \frac{5\alpha + 2\gamma}{24}\left(f_2^2 - f_{12}\right)^2 \\
R_7 & = & \frac{4-f_2^2(0)}{8\alpha}\xi \left(1+\eta_7\right) + \frac{1}{4}f_{12}' \xi' - \frac{1}{8}\left(f_1^2 - f_{12}\right)\left(f_2^2 - f_{12}\right) \\
R_8 & = & - \frac{4\alpha - (\alpha + \gamma) f_2^2(0)}{8\sqrt{3}\alpha^2}\xi (1+\eta_8 ) - \frac{1}{4\sqrt{3}}\left(1-\frac{\gamma}{\alpha}\right) f_{12}' \xi'  \\
& & + \frac{1}{8\sqrt{3}}\left(1 + \frac{\gamma}{\alpha}\right)\left(f_1^2 - f_{12}\right)\left(f_2^2 - f_{12}\right) \\
R_9 & = & - \frac{4\alpha - (\alpha - \gamma) f_2^2(0)}{8\sqrt{3}\alpha^2}\xi (1+\eta_9 ) - \frac{1}{4\sqrt{3}}\left(1-\frac{\gamma}{\alpha}\right) f_{12}' \xi'  \\
& & + \frac{1}{8\sqrt{3}}\left(1 + \frac{\gamma}{\alpha}\right)\left(f_1^2 - f_{12}\right)\left(f_2^2 - f_{12}\right)  \\
(\alpha - \xi) R_{10} & = & - \frac{4-f_2^2(0)}{8}\xi \left(1+ \eta_{10}\right) + \frac{5\alpha -2\gamma -3}{12}f_{12}'\xi' \\
& & - \frac{5\alpha + 2\gamma}{24}\left(f_1^2 - f_{12}\right)\left(f_2^2 - f_{12}\right)
\end{eqnarray*}
where $\eta_i = \eta_i(t)$ are functions in $t$($i=1,\ldots, 10$), with $\eta_i (t) \To 0$ as $t\To 0^+$.
\end{lem}

Next we prove Theorem \ref{thm:introobstruction} in the Introduction.

\begin{proof}[Proof of Theorem \ref{thm:introobstruction}]
We argue by contradiction. Assume that $M^{13}_k$ admits a non-negatively curved invariant metric $g$ with $k \geq 3$. The constant $\beta$ in Proposition \ref{prop:parallelJacobi} and thus $\alpha$ in equation (\ref{eqn:alphagamma}) are determined by the metric $g$. Furthermore, from Proposition \ref{prop:a1squarebounds}, we have $\frac{3}{4}\leq \alpha \leq \frac{7}{12} + \frac{1}{6}\sqrt{13}$.
  
First, note that $\xi(t) > 0$ for $t>0$ by a similar argument as in \cite[Section 2]{GVWZ} and the inequality (\ref{eqn:xipxiupperbound}). From Lemma \ref{lem:a1h1pp}, we have $a_1''(t) \leq 0$ for all $t\in [0, L]$, and it follows that $\xi'(t) = - 2 a_1(t) a_1'(t)\geq 0$ for all $t \in [0,L]$ as $a_1'(0) = 0$. From the inequality (\ref{eqn:xipxiupperbound}) we have
\[
0\leq \frac{f_1\xi'}{\xi}\leq \frac{2\sqrt{6}}{3}(1+\eta(t))
\]
for all $t\in (0, L)$. So the limit superior exists, and we denote
\begin{equation}
\ell = \limsup_{t\To 0+} \frac{f_1\xi'}{\xi} \leq \frac{2\sqrt{6}}{3}.
\end{equation}
Next we will derive a lower bound of $\ell$ from the non-negativity of the curvatures of certain $2$-planes, such that the two bounds contradict to each other if $k > 2$.  

Consider the sectional curvature of the plane spanned by $A_r = X_1 + r X_2$ and $B_q = E_1 + q F_1$:
\[
K(A_r, B_q) = \frac{R(A_r, B_q, B_q, A_r)}{\abs{A_r \wedge B_q}^2}.
\]
Note that a necessary condition for $K(A_r, B_q)\geq 0$ for all $r$, is that the following inequality
\[
I_q = \frac{1}{f_2^4(0)}\left(R(X_1, B_q, B_q, X_1) R(X_2, B_q, B_q, X_2) - R(X_1, B_q, B_q, X_2)^2\right) \geq 0
\]
holds for all $q$. Using the $R_i$'s, we have 
\begin{eqnarray*}
R(X_1, B_q, B_q, X_1) & = & R_1 + 2 q R_2 + q^2 R_3 \\
R(X_2, B_q, B_q, X_2) & = & R_4 + 2 q R_5 + q^2 R_6 \\
R(X_1, B_q, B_q, X_2) & = & R_7 + q\left(R_8 + R_9\right) + q^2 R_{10};
\end{eqnarray*}
and thus
\begin{eqnarray*}
f_2^4(0)I_q & = & \left(R_3 R_6 -R_{10}^2\right) q^4 + 2\left( R_2 R_6 + R_3 R_5 - R_8 R_{10}- R_9 R_{10}\right)q^3 \\
& & +\left[-(R_8 + R_9)^2 - 2R_7 R_{10} + 4 R_2 R_5 + R_1 R_6 + R_3 R_4 \right]q^2 \\
& & +2\left(R_2 R_4 + R_1 R_5 - R_7 R_8- R_7 R_9\right) q + \left(R_1 R_4 - R_7^2\right).
\end{eqnarray*}
Write 
\[
I_q = c_4 q^4 + c_3 q^3 + c_2 q^2 + c_1 q + c_0
\]
with
\begin{eqnarray*}
c_0 & = & f_2^{-4}(0)\left(R_1 R_4 - R_7^2\right) \\
c_1 & = & 2f_2^{-4}(0)\left(R_2 R_4 + R_1 R_5 - R_7 R_8 - R_7 R_9\right) \\
c_2 & = & f_2^{-4}(0)\left(-(R_8 + R_9)^2 - 2R_7 R_{10} + 4 R_2 R_5 + R_1 R_6 + R_3 R_4\right) \\
c_3 &= & 2f_2^{-4}(0)\left( R_2 R_6 + R_3 R_5 - R_8 R_{10}- R_9 R_{10}\right) \\
c_4 & = & f_2^{-4}(0)\left(R_3 R_6 - R_{10}^2\right).
\end{eqnarray*}
From the forms of $R_i$'s in Lemma \ref{lem:Ri}, we have 
\begin{eqnarray*}
c_0 & = & - \frac{1}{16 \alpha}(1+ \eta_{11}) \xi + \frac{1}{16}(1+\eta_{12}) f_1 f_1' \xi' \\
c_1 & = & \frac{2\alpha + \gamma}{8 \sqrt{3}\alpha^2} (1+ \eta_{13})\xi - \frac{1}{4\sqrt{3}}(1+\eta_{14}) f_1 f_1' \xi' \\
c_2 & = & - \frac{\alpha + \gamma}{8 \alpha^2} (1+ \eta_{15}) \xi + \frac{9-2\alpha}{48\alpha} (1+\eta_{16}) f_1 f_1' \xi' \\
c_3 &= & \frac{2\alpha - \gamma}{24\sqrt{3} \alpha^2}(1+\eta_{17})\xi + \frac{2\alpha^2 - 3\alpha + \gamma}{8\sqrt{3}\alpha^2} (1+ \eta_{18}) f_1 f_1' \xi' \\
c_4 & = & \frac{5\alpha + 2\gamma}{48\alpha^2} (1+\eta_{19}) \xi + \frac{3\alpha^2 - \alpha - 2\gamma}{48\alpha^2} (1+\eta_{20}) f_1 f_1' \xi'. 
\end{eqnarray*}
Here $\eta_{11}, \ldots \eta_{20}$ are functions in $t$, with $\eta_i(t) \To 0$ as $t\To 0^+$ for $i = 11, \ldots, 20$. One can verify the forms of $c_0, \ldots, c_4$ above in the following two steps:
\begin{enumerate}[(i)]
\item Check the fact that the term without the factor $\xi$ or $\xi'$ in each $c_i$ vanishes. 
\item Calculate the leading term with factor $\xi$ or $\xi'$ in each $c_i$.
\end{enumerate}

Take the sequence $\set{t_n} \subset (0, L)$ with $\lim_{n\To \infty} t_n = 0$ and 
\[
\ell = \lim_{n\To \infty} \frac{f_1(t_n)\xi'(t_n)}{\xi(t_n)}.
\]
Note that the coefficients in $c_i$'s appear in the quartic functions $\Psi_1$ and $\Psi_2$ in Lemma \ref{lem:quarticfunctions}. For any fixed $q$ we take the limit of $\xi^{-1} I_q$ along the sequence $\set{t_n}$ and it follows that 
\begin{equation}\label{eqn:Psi12ell}
0 \leq  \Psi_1(q) + \Psi_2(q) f_1'(0) \ell = \Psi_1(q) + \Psi_2(q) \frac{4}{k\sqrt{6}}\ell.
\end{equation}
From Lemma \ref{lem:quarticfunctions}, there is a real number $q_\alpha$ such that 
\[
\Psi_1(q_\alpha) = - \frac{4}{3}\Psi_2(q_\alpha)\quad \text{and}\quad \Psi_2(q_\alpha) > 0.
\]
Letting $q = q_\alpha$ in the inequality (\ref{eqn:Psi12ell}) yields 
\begin{eqnarray*}
0 & \leq & - \frac{4}{3}\Psi_2(q_\alpha) + \Psi_2(q_\alpha) \frac{4}{k \sqrt 6} \frac{2\sqrt 6}{3} \\
& \leq & \left(\frac{8}{3k} - \frac{4}{3}\right)\Psi_2(q_\alpha)
\end{eqnarray*}
and so we have $k \leq 2$. It contradicts to the assumption that $k \geq 3$, and we finish the proof.
\end{proof}
\begin{rem}\label{rem:shortproof}
There is a relatively shorter proof that works for $k \geq 5$: Instead we consider the sectional curvature of the $2$-plane spanned by $A_r = X_1 + r X_2$ and $B = E_1$, i.e., fix $q=0$. Then $K(A_r, B)\geq 0$ implies that $I_0\geq 0$, i.e.,
\[
c_0 = -\frac{1}{16\alpha} \left(1+\eta_{11}\right)\xi + \frac{1}{16}\left(1+\eta_{12}\right) f_1 f_1' \xi' \geq 0.
\]
It follows that 
\[
\frac{f_1 \xi'}{\xi} \geq \frac{1+\eta_{11}}{1+\eta_{12}}\frac{1}{\alpha f_1'}
\]
when $t>0$ small. Taking the limit $t_n \To 0$ yields
\[
\ell \geq \frac{1}{\alpha} \frac{k \sqrt 6}{4}.
\]
Combine with the inequality (\ref{lem:lupperbound}), and we obtain
\[
\frac{2\sqrt{6}}{3}\geq \ell \geq  \frac{1}{\alpha}\frac{k \sqrt 6}{4}.
\]
From Proposition \ref{prop:a1squarebounds}, we have the following estimate: 
\[
k \leq \frac{8}{3} \alpha \leq \frac{8}{3}\left(\frac{7}{12} + \frac{\sqrt{13}}{6} \right) \approx 3.16.
\]
However this short proof does not rule out the case $k = 3$. 
\end{rem}

\medskip{}
\appendix

\section{The computations of Riemann curvature tensors}

In this section we collect the detailed computations of Riemann curvature tensors which are used in Section 5 and 6: Proposition \ref{prop:a1squarebounds}, Lemmas \ref{lem:lupperbound} and \ref{lem:Ri}. The formulas of Riemann curvature tensors on a cohomogeneity one manifold have been derived in \cite{GroveZillerposRicci}. Write $R(X,Y,Z,W) = g(R(X,Y)Z, W)$, and the convention of the sectional curvature is given by 
\[
K(X,Y) = \frac{R(X, Y, Y,X)}{\abs{X\wedge Y}^2}
\]
for a $2$-plane spanned by $X$ and $Y$. Recall that $Q$ is a fixed bi-invariant inner product on $\frakg = \so(2)+\frakg_2$, and $\frakp =\frakh^{\perp}$ where $\frakh$ is the Lie algebra of the principal isotropy subgroup $\sfH$. The invariant metric is $g = dt^2 + g_t$, and
\[
g_t (X^*, Y^*) = Q(P X, Y)
\]
where $X^*$ and $Y^*$ are Killing vector field generated by $X, Y \in \frakp$ along the normal geodesic $c(t)$, and $P = P(t) : \frakp \To \frakp$ is a family of positive definite $\Ad_{\sfH}$-invariant endomorphisms for $t \in (0, L)$. In terms of the $Q$-orthonormal basis 
\[\set{X_1, X_2, Y_1, Y_2, E_1, \ldots, E_4, F_1, \ldots, F_4}
\] 
we have
\begin{eqnarray*}
P X_1 & = & f_1^2(t) X_1 + f_{12}(t) X_2 \\
P X_2 & = & f_{12}(t) X_1 + f_2^2(t) X_2 \\
P Y_1 & = & h_1^2(t) Y_1 \\
P Y_2 & = & Y_2 \\
P E_i & = & a_1^2(t) E_i - \beta a_1^2(t) F_i \\
P F_i & = & -\beta a_1^2(t) E_i + (\beta^2 a_1^2(t) + 1) F_i
\end{eqnarray*}
with $1\leq i \leq 4$. The following two bilinear maps are defined in \cite{Puettmann}:
\begin{equation}
B_{\pm} = \frac{1}{2}\left([X, P Y] \mp [P X, Y]\right).
\end{equation}
Here $B_+$ is symmetric with $B_+(X, Y)\in \frakp$ for any $X, Y \in \frakp$, and $B_-$ is skew-symmetric. The formulas of Riemann curvature tensors in terms of $Q$, $P_t$ and $B_{\pm}$ are given in Proposition 1.9 and Corollary 1.10 in \cite{GroveZillerposRicci}. The following special case of formula 1.9(a) in \cite{GroveZillerposRicci} is also useful. For any $X, Y, Z \in \frakp$ we have 
\begin{eqnarray*}
R(X,Y,Z,X) & = & \frac 12 Q\pset{\Bm(X, Y), [X,Z]} + \frac 1 2 Q\pset{[X, Y], \Bm(X, Z)} \\
& & -\frac 12 Q\pset{P[X, Y]_{\frakp}, [X, Z]_\frakp} - \frac 1 4 Q\pset{P[X,Z]_{\frakp}, [X, Y]_\frakp} \\
& & + Q\pset{\Bp(X,Z), P^{-1} \Bp(X, Y)} - Q\pset{\Bp(X,X), P^{-1} \Bp(Y, Z)} \\
& & + \frac 1 4 Q\pset{P'(t) X, Z} Q\pset{P'(t) X, Y} - \frac 1 4 Q\pset{P'(t) X, X} Q\pset{P'(t) Y, Z}.
\end{eqnarray*}
Recall the constants 
\[
\alpha =a_1^2(0) = \frac{\sqrt 3}{\sqrt 3\left(1- \beta^2\right) + 2\beta}
\]
and $\gamma = \sqrt{\alpha(4\alpha - 3)}$ in equations (\ref{eqn:a1sqaurebeta}) and (\ref{eqn:alphagamma}).

\subsection{The Riemann curvature tensors in Proposition \ref{prop:a1squarebounds}}
First we have 
\[
[Y_1, E_1] = \sqrt{3}E_2 \quad\text{and}\quad [Y_1, F_1] = - \frac{1}{\sqrt 3} F_2.
\]
Then the bilinear maps are given by 
\begin{eqnarray*}
2\Bm(Y_1, E_1) & = & [Y_1, P(0) E_1] + [P(0) Y_1, E_1] \\
& = & [Y_1, \alpha E_1 -\alpha \beta F_1] + [Y_1, E_1] \\
& = & \sqrt 3 (\alpha + 1)E_2 + \frac{\alpha \beta}{\sqrt 3} F_2 \\
2\Bp(Y_1, E_1) & = & [Y_1, P(0) E_1] - [P(0) Y_1, E_1] \\
& = & \sqrt 3 (\alpha -1)E_2 + \frac{\alpha \beta}{\sqrt 3} F_2
\end{eqnarray*}
and 
\begin{eqnarray*}
2\Bm(Y_1, F_1) & = & [Y_1, P(0) F_1] + [P(0)Y_1, F_1] \\
& = & [Y_1, -\alpha \beta E_1 + (\alpha \beta^2 +1)F_1] +[Y_1, F_1] \\
& = & -\sqrt 3 \alpha \beta E_2 -\frac{\alpha \beta^2 + 2}{\sqrt 3} F_2 \\
2\Bp(Y_1, F_1) & = & [Y_1, P(0) F_1] - [P(0) Y_1, F_1] \\
& = & -\sqrt 3 \alpha \beta E_2 -\frac{\alpha \beta^2}{\sqrt 3} F_2 .
\end{eqnarray*}
It follows that 
\begin{eqnarray*}
P^{-1}(0) \Bp(Y_1, E_1) & = & \frac{\sqrt 3(\alpha -1)}{2} P^{-1}(0) E_2 + \frac{\alpha \beta}{2\sqrt 3} P^{-1}(0) F_2 \\
& = & \frac{\sqrt 3(\alpha -1)}{2} \pset{\pset{\beta^2 + \frac{1}{\alpha}}E_2 + \beta F_2} + \frac{\alpha \beta}{2\sqrt 3}\pset{\beta E_2 + F_2} \\
& = & \frac{\sqrt{3}\beta \pset{-1+\beta^2}}{\sqrt 3(1-\beta^2)+2\beta} E_2 + \frac{(-3 + 4\alpha)\beta}{2\sqrt 3} F_2,
\end{eqnarray*}
and 
\begin{eqnarray*}
P^{-1}(0) \Bp(Y_1, F_1) & = & -\frac{\sqrt 3 \alpha \beta}{2} P^{-1}(0) E_2 - \frac{\alpha \beta^2}{2\sqrt 3} P^{-1}(0) F_2 \\
& = & -\frac{\sqrt 3\alpha \beta}{2} \pset{\pset{\beta^2 + \frac{1}{\alpha}}E_2 + \beta F_2} - \frac{\alpha \beta^2}{2\sqrt 3}\pset{\beta E_2 + F_2} \\ 
& = & \frac{\beta (\beta + \sqrt 3)^2}{-2\sqrt 3(1-\beta^2) -4 \beta}E_2 - \frac{2\alpha \beta^2}{\sqrt 3}F_2.
\end{eqnarray*}

Note that $\Bp(Y_1, Y_1) = [Y_1, P(0) Y_1] = 0$. So one can compute the three Riemann curvature tensors as follows:
\begin{eqnarray*}
R(Y_1, E_1, E_1, Y_1) & = & \frac{3(1+\alpha)}{2}-\frac 3 4 Q \pset{\sqrt 3 \alpha E_2 - \sqrt 3 \alpha \beta F_2,\sqrt 3 E_2} \\
& & + Q\pset{\frac{\sqrt 3(\alpha -1)}{2}E_2 + \frac{\alpha \beta}{2\sqrt 3}F_2, \frac{\sqrt{3}\beta \pset{-1+\beta^2}}{\sqrt 3(1-\beta^2)+2\beta} E_2 + \frac{(-3 + 4\alpha)\beta}{2\sqrt 3} F_2} \\
& = & \frac{3(1+\alpha)}{2}-\frac{9\alpha}{4} + \frac{\sqrt 3(\alpha -1)}{2}\frac{\sqrt{3}\beta \pset{-1+\beta^2}}{\sqrt 3(1-\beta^2)+2\beta} + \frac{\alpha \beta}{2\sqrt 3}  \frac{(-3 + 4\alpha)\beta}{2\sqrt 3} \\
& = & \frac{6\sqrt{3} \beta^5 + 9 \beta^4 - 32\sqrt{3} \beta^3 + 10 \beta^2 + 18 \sqrt{3}\beta + 9}{4\left(\sqrt 3 \beta^2 - 2\beta - \sqrt 3\right)^2},
\end{eqnarray*}
\begin{eqnarray*}
R(Y_1, F_1, F_1, Y_1) & = & \frac{\alpha \beta^2 +2}{6} - \frac{3}{4}\cdot\frac{1}{3}Q\pset{P(0)F_2, F_2} \\
& & + Q\pset{ -\frac{\sqrt 3 \alpha \beta}{2} E_2 -\frac{\alpha \beta^2}{2\sqrt 3} F_2,  \frac{\beta (\beta + \sqrt 3)^2}{-2\sqrt 3(1-\beta^2) -4 \beta}E_2 - \frac{2\alpha \beta^2}{\sqrt 3}F_2} \\
& = & \frac{\alpha \beta^2 +2}{6} -\frac{\alpha\beta^2 + 1}{4} + \frac{\sqrt 3 \alpha \beta^2 (\beta + \sqrt 3)^2}{4\sqrt 3(1-\beta^2) + 8 \beta} + \frac{\alpha^2 \beta^4}{3} \\
& = & \frac{27 \beta^4 + 12 \sqrt{3}\beta^3 + 22\beta^2 + 4\sqrt{3}\beta + 3}{12\left(\sqrt 3 \beta^2 - 2\beta - \sqrt 3\right)^2},
\end{eqnarray*}
and 
\begin{eqnarray*}
R(Y_1, E_1, F_1,Y_1) & = & \frac 12 Q\pset{\Bm(Y_1, E_1), [Y_1,F_1]} + \frac 1 2 Q\pset{[Y_1, E_1], \Bm(Y_1, F_1)} \\
& & -\frac 12 Q\pset{P(0)[Y_1, E_1]_{\frakp}, [Y_1, F_1]_\frakp} - \frac 1 4 Q\pset{P(0)[Y_1,F_1]_{\frakp}, [Y_1, E_1]_\frakp} \\
& & + Q\pset{\Bp(Y_1, F_1), P^{-1}(0) \Bp(Y_1, E_1)} \\
& & - Q\pset{\Bp(Y_1,Y_1), P^{-1}(0) \Bp(E_1, F_1)} \\
& = & \frac 1 4 Q\pset{\sqrt 3 (\alpha + 1)E_2 + \frac{\alpha \beta}{\sqrt 3} F_2, - \frac{1}{\sqrt 3} F_2} \\
& & + \frac 1 4 Q\pset{\sqrt 3 E_2, -\sqrt 3 \alpha \beta E_2 -\frac{\alpha \beta^2 + 2}{\sqrt 3} F_2} \\
& & - \frac 1 2 Q\pset{\sqrt 3 P(0) E_2,-\frac{1}{\sqrt 3}F_2} - \frac1 4 Q\pset{-\frac{1}{\sqrt 3}P(0) F_2, \sqrt 3 E_2} \\
& & + \frac 1 2 Q\pset{ -\sqrt 3 \alpha \beta E_2 -\frac{\alpha \beta^2}{\sqrt 3} F_2,  \frac{\sqrt{3}\beta \pset{-1+\beta^2}}{\sqrt 3(1-\beta^2)+2\beta} E_2 + \frac{(-3 + 4\alpha)\beta}{2\sqrt 3} F_2} \\
& = & - \frac{1}{12}\alpha \beta - \frac 3 2 \alpha \beta + \frac 1 2 \pset{\frac{-3\alpha \beta^2(-1+\beta^2)}{\sqrt 3(1-\beta^2)+ 2\beta} - \frac{(-3+4\alpha)\alpha \beta^3}{6}} \\
& = & - \frac{\beta\left(9 \beta^4 + 12\sqrt{3}\beta^3 - 54 \beta^2 + 20\sqrt{3}\beta + 57\right)}{12\left(\sqrt 3 \beta^2 - 2\beta - \sqrt 3\right)^2}.
\end{eqnarray*}

\subsection{The curvature formula in Lemma \ref{lem:lupperbound}}
Recall that $X = E_1 - \sqrt 3 F_1$ and $Y = \sqrt 3 E_4 + F_4$. First note that $[X, Y] = 0$. The images under $P=P(t)$ are given by
\begin{eqnarray*}
P X & = & P E_1 - \sqrt 3 P F_1 \\
& = & a_1^2 E_1 - \beta a_1^2 F_1 -\sqrt 3 \pset{-\beta a_1^2 E_1 + (\beta^2 a_1^2 +1)F_1} \\
& = & a_1^2 (1+\sqrt 3 \beta) E_1 - \pset{\beta a_1^2 + \sqrt 3 \beta^2 a_1^2 + \sqrt 3}F_1 \\
P Y & = &\sqrt 3 P E_4 + P F_4 \\
& = & \sqrt 3\pset{a_1^2 E_4 - \beta a_1^2 F_4} + \pset{-\beta a_1^2 E_4 + (\beta^2 a_1^2 + 1)F_4} \\
& = & a_1^2(\sqrt 3-\beta)E_4 + \pset{\beta^2 a_1^2 - \sqrt 3 \beta a_1^2 + 1}F_4.
\end{eqnarray*}
Note that $[E_1, F_1] = [E_4, F_4] = [E_1, E_4]_\frakp = 0$, and 
\[
[E_1, F_4] = -\frac{1}{\sqrt 2}(X_1 - X_2), \quad [E_4, F_1] = \frac{1}{\sqrt 2}(X_1 - X_2), \quad [F_1, F_4]_\frakp = \frac{\sqrt{6}}{3}\pset{X_1 - X_2}.
\]
It follows that the bilinear maps are $\Bp(X,X) = \Bp(Y, Y) = 0$, and 
\begin{eqnarray*}
\Bp(X, Y) & = & [X, PY] - [PX, Y] \\
& = & \frac{\sqrt 2 \pset{-3 + \pset{-3\beta^2 + 2\sqrt 3 \beta + 3}a_1^2}}{3}\pset{X_1 - X_2}.
\end{eqnarray*}
So there are only two non-vanishing terms in $R(X, Y, Y, X)$ that yield 
\begin{eqnarray*}
R(X, Y, Y, X) & = & Q\pset{\Bp(X, Y), P^{-1}\Bp(X, Y)} - \frac{1}{4}Q\pset{P'(t) X, X} Q\pset{P'(t) Y, Y} \\
& = & \frac{2 \pset{-3 + \pset{-3\beta^2 + 2\sqrt 3 \beta + 3}a_1^2}^2}{9}\frac{f_1^2 + f_2^2 + 2f_{12}}{f_1^2 f_2^2 - f_{12}^2}  \\
& & + \pset{-3+2 \sqrt{3} \beta -\beta ^2} \pset{1+2 \sqrt{3} \beta +3 \beta ^2} a_1^2 \pset{a_1'}^2.
\end{eqnarray*}
After the substitutions $\xi = \alpha - a_1^2$ and $\beta$ in terms of $\alpha$, we have
\[
R(X,Y,Y,X) = \frac{2}{\alpha^2} \frac{f_1^2 + f_2^2 + 2f_{12}}{f_1^2 f_2^2 - f_{12}^2}\xi^2 - \frac{3}{4\alpha^2}\pset{\xi'}^2
\]
that gives the formula in equation (\ref{eqn:KXYxi}).

\subsection{The Riemann curvature tensors $R_1,\ldots, R_{10}$ in Lemma \ref{lem:Ri}}

Similar to the previous sections A.1 and A.2,  a straightforward but tedious computation shows the following formulas, which are used to derive Lemma \ref{lem:Ri}.

\begin{prop}
We have
\begin{eqnarray*}
R_1 & = & -\frac{\xi}{2\alpha} + \frac{\xi^2}{8\alpha^2} + \frac{\xi f_1^2}{4\alpha} + \frac{1}{8}f_1^4 - \frac{\xi f_{12}}{4\alpha} - \frac{1}{4}f_1^2 f_{12} + \frac{1}{8} f_{12}^2 + \frac{1}{2} f_1 f_1' \xi' \\
\sqrt 3 R_2 & = &  \frac{\xi}{2 \alpha} - \frac{\xi^2}{8 \alpha^2} + \frac{\gamma}{8\alpha^3}\xi^2 - \frac{f_1^2}{4\alpha}\xi - \frac{f_1^4}{8} - \frac{\gamma f_1^4}{8\alpha} + \frac{f_{12}}{4\alpha}\xi + \frac{f_1^2 f_{12}}{4} + \frac{\gamma f_1^2 f_{12}}{4\alpha} \\
& & - \frac{f_{12}^2}{8} - \frac{\gamma f_{12}^2}{8\alpha} - \frac{f_1}{2}f_1' \xi' + \frac{\gamma f_1}{2\alpha}f_1' \xi' \\
(\alpha -\xi) R_3 & = & \frac{\xi}{2} - \frac{7}{24\alpha}\xi^2 -\frac{\gamma}{12\alpha^2}\xi^2 + \frac{1}{8\alpha^3}\xi^3 - \frac{5}{24\alpha^2} \xi^3 + \frac{\gamma}{12\alpha^3}\xi^3 - \frac{f_1^2}{4}\xi - \frac{f_1^2}{4\alpha^2}\xi^2 \\
& & + \frac{f_1^2}{4\alpha}\xi^2 + \frac{5\alpha}{24}f_1^4 + \frac{\gamma}{12}f_1^4 - \frac{5f_1^4}{24}\xi + \frac{f_1^4}{8\alpha}\xi - \frac{\gamma f_1^4}{12\alpha}\xi + \frac{f_{12}}{4}\xi + \frac{f_{12}}{4\alpha^2} \xi^2 \\ 
& & - \frac{f_{12}}{4\alpha}\xi^2 - \frac{5\alpha f_1^2 f_{12}}{12} - \frac{\gamma f_1^2 f_{12}}{6} + \frac{5 f_1^2 f_{12}}{12}\xi - \frac{f_1^2 f_{12}}{4\alpha}\xi + \frac{\gamma f_1^2 f_{12}}{6\alpha} \xi \\
& & + \frac{5\alpha f_{12}^2}{24} + \frac{\gamma f_{12}^2}{12} - \frac{5f_{12}^2}{24}\xi + \frac{f_{12}^2}{8\alpha}\xi - \frac{\gamma f_{12}^2}{12\alpha} \xi - \frac{f_1}{2} f_1' \xi'  + \frac{5\alpha f_1}{6}f_1' \xi' \\
& & - \frac{\gamma f_1}{3}f_1' \xi' - \frac{5}{6}f_1 \xi f_1' \xi'  + \frac{1}{2\alpha} f_1 \xi f_1' \xi' + \frac{\gamma}{3\alpha} f_1 \xi f_1' \xi' .
\end{eqnarray*}
$R_4$, $R_5$ and $R_6$ can be obtained from $R_1$, $R_2$ and $R_3$ respectively by switching $f_1$ and $f_2$. 
\begin{eqnarray*}
R_7 & = & \frac{1}{2\alpha} \xi - \frac{1}{8\alpha^2}\xi^2 - \frac{f_1^2}{8\alpha}\xi - \frac{f_2^2}{8 \alpha} \xi - \frac{1}{8}f_1^2 f_2^2 + \frac{f_{12}}{4\alpha} \xi + \frac{1}{8}f_1^2 f_{12} + \frac{1}{8}f_2^2 f_{12} - \frac{1}{8} f_{12}^2 + \frac{1}{4}f_{12}' \xi'  \\
\sqrt{3} R_8 & = & - \frac{1}{2\alpha}\xi + \frac{1}{8\alpha^2} \xi^2 - \frac{\gamma}{8\alpha^3}\xi^2 + \frac{f_1^2}{8\alpha} \xi - \frac{\gamma f_1^2}{8\alpha^2}\xi + \frac{f_2^2}{8\alpha}\xi + \frac{\gamma f_2^2}{8\alpha^2}\xi + \frac{f_1^2 f_2^2}{8} + \frac{\gamma f_1^2 f_2^2}{8\alpha} \\
& & - \frac{f_{12}}{4\alpha}\xi - \frac{f_1^2 f_{12}}{8} - \frac{\gamma f_1^2 f_{12}}{8\alpha} - \frac{f_2^2 f_{12}}{8} - \frac{\gamma f_2^2 f_{12}}{8\alpha} + \frac{f_{12}^2}{8} + \frac{\gamma f_{12}^2}{8\alpha} - \frac{1}{4}f_{12}' \xi' + \frac{\gamma}{4\alpha} f_{12}' \xi' \\
\sqrt{3}R_9 & = & \sqrt{3}R_8 + \frac{\gamma f_1^2}{4\alpha^2} \xi - \frac{\gamma f_2^2}{4\alpha^2}\xi 
\end{eqnarray*}
and 
\begin{eqnarray*}
(\alpha - \xi) R_{10} & = & - \frac{1}{2}\xi + \frac{7}{24\alpha}\xi^2 + \frac{\gamma}{12\alpha^2}\xi^2 - \frac{1}{8\alpha^3}\xi^3 + \frac{5}{24\alpha^2} \xi^3 - \frac{\gamma}{12 \alpha^3}\xi^3 + \frac{f_1^2}{8}\xi + \frac{f_1^2}{8\alpha^2}\xi^2 \\
& & - \frac{f_1^2}{8\alpha} \xi^2 + \frac{f_2^2}{8}\xi + \frac{f_2^2}{8\alpha^2} \xi^2 - \frac{f_2^2}{8\alpha} \xi^2 - \frac{5\alpha f_1^2 f_2^2}{24} - \frac{\gamma f_1^2 f_2^2}{12} + \frac{5 f_1^2 f_2^2}{24}\xi - \frac{f_1^2 f_2^2}{8\alpha} \xi \\
& & + \frac{\gamma f_1^2 f_2^2}{12\alpha}\xi - \frac{f_{12}}{4}\xi - \frac{f_{12}}{4\alpha^2}\xi^2 + \frac{f_{12}}{4\alpha}\xi^2 + \frac{5\alpha f_1^2 f_{12}}{24} + \frac{\gamma f_1^2 f_{12}}{12} - \frac{5f_1^2 f_{12}}{24}\xi \\
& & + \frac{f_1^2 f_{12}}{8\alpha}\xi - \frac{\gamma f_1^2 f_{12}}{12\alpha}\xi + \frac{5\alpha f_2^2 f_{12}}{24} + \frac{\gamma f_2^2 f_{12}}{12} - \frac{5f_2^2 f_{12}}{24}\xi + \frac{f_2^2 f_{12}}{8\alpha}\xi \\
& & - \frac{\gamma f_2^2 f_{12}}{12\alpha} \xi -\frac{5\alpha f_{12}^2}{24} - \frac{\gamma f_{12}^2}{12} + \frac{5f_{12}^2}{24}\xi - \frac{f_{12}^2}{8\alpha}\xi + \frac{\gamma f_{12}^2}{12\alpha} \xi \\
& & - \frac{1}{4}f_{12}' \xi' + \frac{5\alpha}{12} f_{12}' \xi' - \frac{\gamma}{6}f_{12}' \xi' - \frac{5}{12} \xi f_{12}' \xi' + \frac{1}{4\alpha} \xi f_{12}' \xi' + \frac{\gamma}{6\alpha} \xi f_{12}' \xi'.
\end{eqnarray*}
\end{prop}

\bigskip



\begin{thebibliography}{XXXX9}

\bibitem[ADPR]{ADPR} U. Abresch, C. Dur\'{a}n, T. P\"{u}ttmann and A. Rigas, \emph{Wiedersehen metrics and exotic involutions of Euclidean spheres}, J. Reine Angew. Math. \textbf{605}(2007), 1--21.

\bibitem[AB]{AtiyahBott} M. F. Atiyah and R. Bott, \emph{A Lefschetz fixed point formula for elliptic complexes. II. Applications}, Ann. of Math. (2) \textbf{88}(1968), 451--491.

\bibitem[BH]{BackHsiang} A. Back and W.-Y. Hsiang, \emph{Equivariant geometry and Kervaire spheres}, Trans. Amer. Math. Soc. \textbf{304}(1987), no. 1, 207--227.

\bibitem[Ba]{Baez} J. Baez, \emph{The octonions}, (English summary) Bull. Amer. Math. Soc. (N.S.) \textbf{39}(2002), no. 2, 145--205, and \emph{Errata for: ``The octonions''}, Bull. Amer. Math. Soc. (N.S.) \textbf{42}(2005), no. 2, 213.

\bibitem[Br]{Brieskorn} E. Brieskorn, \emph{Beispiele zur Differentialtopologie von Singularit\"{a}ten}. (German) Invent. Math. \textbf{2}(1966), 1--14. 

\bibitem[Da]{Davis} M. W. Davis, \emph{Some group actions on homotopy spheres of dimension seven and fifteen}, Amer. J. Math. \textbf{104}(1982), no. 1, 59--90.

\bibitem[DP]{DuranPuettmann} C. Dur\'{a}n and T. P\"{u}ttmann, \emph{A minimal Brieskorn 5-sphere in the Gromoll-Meyer sphere and its applications}, Michigan Math. J. \textbf{56}(2008), no. 2, 419--451.

\bibitem[Gi]{Giffen} C. H. Giffen, \emph{Smooth homotopy projective spaces}, Bull. Amer. Math. Soc. \textbf{75}(1969), 509--513.

\bibitem[GM]{GromollMeyer} D. Gromoll and W. Meyer, \emph{An exotic sphere with nonnegative sectional curvature}, Ann. of Math. (2) \textbf{100}(1974), 401--406.

\bibitem[GVWZ] {GVWZ} K. Grove, L. Verdiani, B, Wilking and W. Ziller, \emph{Non-negative curvature obstruction in cohomogeneity one and the Kervaire spheres}, Ann. del. Scuola Norm. Sup. \textbf{5}(2006), 159--170.

\bibitem[GWZ]{GroveWilkingZiller} K. Grove, B. Wilking and W. Ziller, \emph{Positively curved cohomogeneity one manifolds and 3-Sasakian geometry}, J. Differential Geom. \textbf{78}(2008), no. 1, 33--111.

\bibitem[GZ1]{GroveZillerMilnorsphere} K. Grove and W. Ziller, \emph{Curvature and symmetry of Milnor spheres}, Ann. of Math. (2) \textbf{152}(2000), no. 1, 331--367.

\bibitem[GZ2]{GroveZillerposRicci} K. Grove and W. Ziller, \emph{Cohomogeneity one manifolds with positive Ricci curvature}, Invent. Math. \textbf{149}(2002), no. 3, 619--646.

\bibitem[He]{Heobstruction} C. He, \emph{New examples of obstructions to non-negative sectional curvatures in cohomogeneity one manifolds}, Trans. Amer. Math. Soc. \textbf{ 366}(2014), no. 11, 6093--6118.

\bibitem[HM]{HirschMilnor} M. W. Hirsch and J. Milnor, \emph{Some curious involutions of spheres}, Bull. Amer. Math. Soc. \textbf{70}(1964) 372--377.

\bibitem[HH]{HsiangHsiang} W.-C,. Hsiang and W.-Y. Hsiang, \emph{On compact subgroups of the diffeomorphism groups of Kervaire spheres},  Ann. of Math. (2) \textbf{85}(1967), 359--369.

\bibitem[Mi]{Milnor7sphere} J. Milnor, \emph{On manifolds homeomorphic to the {$7$}-sphere}, Ann. of Math. (2) \textbf{ 64}(1956), 399--405.

\bibitem[Mu]{Murakami}  S. Murakami, \emph{Exceptional simple Lie groups and related topics in recent differential geometry}, Differential geometry and topology (Tianjin, 1986--87), 183--221, Lecture Notes in Math., \textbf{1369}, Springer, Berlin, 1989. 

\bibitem[Pu]{Puettmann} T. P\"{u}ttmann, \emph{Optimal pinching constants of odd-dimensional homogeneous spaces}, Invent. Math. \textbf{138}(1999), no. 3, 631--684.

\bibitem[RW]{RajanWilhelm} P. Rajan and F. Wilhelm, \emph{Almost nonnegative curvature on some fake $\RP^{6}$s and $\RP^{14}$s}, arXiv:1510.05320v2 [math.DG], preprint, 2015.  

\bibitem[ST]{SchwachTuschmann} L. J. Schwachh\"{o}fer and W. Tuschmann, \emph{Metrics of positive Ricci curvature on quotient spaces}, Math. Ann. \textbf{330}(2004), no. 1, 59--91.

\bibitem[Sh]{Shimada} N. Shimada, \emph{Differentiable structures on the $15$-sphere and Pontrjagin classes of certain manifolds}, Nagoya Math. J. \textbf{12}(1957) 59--69.

\bibitem[St]{Straume} E. Straume, \emph{Compact connected Lie transformation groups on spheres with low cohomogeneity. I}, Mem. Amer. Math. Soc. \textbf{119}(1996), no. 569, vi+93 pp.

\bibitem[VZ]{VerdianiZiller} L. Verdiani and W. Ziller, \emph{Concavity and rigidity in non-negative curvature}, J. Differential Geom. \textbf{97}(2014), no. 2, 349--375.

\bibitem[WZ]{WilkingZiller} B. Wilking and W. Ziller, \emph{Revisiting homogeneous spaces with positive curvature}, arXiv:1503.06256v1 [math.DG], 2015. To appear in J. Reine Angew. Math..

\end{thebibliography}
\end{document}